\documentclass[11pt]{article}
\usepackage[T1]{fontenc}
\usepackage{latexsym,amssymb,amsmath,amsfonts,amsthm}
\usepackage{graphicx}
\usepackage{color}
\date{}

\begin{document}

\title{Uniqueness and stability for the recovery of a time-dependent source in elastodynamics}

\author{Guanghui Hu  and  Yavar Kian\\
\footnotesize Beijing Computational Science Research Center, Beijing 100193, China\\ [-1mm]
\footnotesize Aix Marseille Univ, Universit\'e de Toulon, CNRS, CPT, Marseille, France\\ [-1mm]
\footnotesize E-Mail: hu@csrc.ac.cn \ \ yavar.kian@univ-amu.fr}

\maketitle


\newtheorem{thm}{Theorem}
\newtheorem{cor}[thm]{Corollary}
\newtheorem{lem}{Lemma}
\newtheorem{prop}[thm]{Proposition}
\newtheorem{defn}{Definition}
\newtheorem{rem}{Remark}
\numberwithin{equation}{section}
\newcommand{\norm}[1]{\left\Vert#1\right\Vert}
\newcommand{\abs}[1]{\left\vert#1\right\vert}
\newcommand{\To}{\longrightarrow}
\newcommand{\BX}{\mathbf{B}(X)}
\newcommand{\cB}{\mathcal{B}}
\newcommand{\cC}{\mathcal{C}}
\newcommand{\cD}{\mathcal{D}}
\newcommand{\re}{\mathfrak R}
\newcommand{\cK}{\mathcal{K}}
\newcommand{\cM}{\mathcal{M}}
\newcommand{\cN}{\mathcal{N}}
\newcommand{\cP}{\mathcal{P}}
\newcommand{\oo}{\boldsymbol{0}}
\newcommand{\ve}{\varepsilon}
\newcommand{\pp}{\pmb{|}}
\newcommand{\be}{\begin{eqnarray}}
\newcommand{\ben}{\begin{eqnarray*}}
\newcommand{\en}{\end{eqnarray}}
\newcommand{\enn}{\end{eqnarray*}}
\newcommand{\benn}{\begin{equation}}
\newcommand{\een}{\end{equation}}
\newcommand{\hx}{\hat{x}}
\newcommand{\al}{\alpha}
\newcommand{\ga}{\gamma}
\newcommand{\G}{\Gamma}
\newcommand{\Om}{\Omega}
\newcommand{\N}{{\mathbb{N}}}
\newcommand{\Z}{{\mathbb{Z}}}
\newcommand{\R}{{\mathbb{R}}}
\newcommand{\BS}{{\mathbb{S}}}
\newcommand{\Lam}{{\Lambda}}
\newcommand{\C}{{\mathbb{C}}}
\newcommand{\bx}{{\rm {x}}}
\newcommand{\supp}{{\rm {supp}~}}
\newcommand{\grad}{{\rm {grad}~}}
\newcommand{\Div}{{\rm {div}~}}
\newcommand{\curl}{{\rm {curl}~}}
\newcommand{\imaganary}{\mbox{Im}}
\newcommand{\real}{\mbox{Re}}
\newcommand{\tC}{{\tilde{C}}}
\definecolor{rot}{rgb}{0.000,0.000,0.000}
\newcommand{\tcr}{\textcolor{rot}}

\begin{abstract}
This paper is concerned with inverse source problems for the time-dependent Lam\'e system  in an unbounded domain  corresponding to  the exterior of a bounded cavity or the full space $\R^3$. If the time and spatial variables of the source term can be separated with compact support, we prove that the vector valued spatial source term can be uniquely determined by boundary Dirichlet data in the exterior of a given cavity. Uniqueness and stability for recovering some class of time-dependent source terms are also obtained using partial boundary data.


\vspace{.2in} {\bf Keywords:} Linear elasticity, inverse source problems, time domain, uniqueness, stability estimate.
\end{abstract}

\section{Introduction}
\subsection{Statement of the problem}
Consider the radiation of an elastic source $F$ outside a cavity $D$ described by the system
\be\label{llame1}\rho \partial_{tt}U(x,t)=\mathcal L_{\lambda,\mu} U(x,t)+F(x,t), \quad x=(x_1,x_2,x_3)\in \R^3\backslash{\overline{D}},\; t>0\en
where $\rho$ denotes the density, \tcr{$\mu$ and $\lambda$ the Lam\'e coefficients}, $U=(u_1,u_2,u_3)^\top$ the displacement vector, $D\subset\R^3$ the region of the cavity
and $\mathcal{L}_{\lambda,\mu}U$ the Lam\'e operator defined by
\be\label{lame2}
\mathcal{L}_{\lambda,\mu}U:=-\mu(x) \nabla\times\nabla\times U+
(\lambda(x)+2\mu(x))\nabla\nabla\cdot U+(\nabla\cdot U)\nabla\lambda(x)+((\nabla U)+(\nabla U)^T)\nabla\mu(x).
\en
Throughout the paper, it is supposed that \tcr{$\rho>0$ is a constant}
and $\mu,\lambda\in \mathcal C^3(\R^3)$ satisfy $\mu>0$, $\lambda\geq0$.
Further, the density function and the Lam\'e coefficients are supposed to be constants in $|x|>R$ for some sufficiently large $R>0$ such that $D\subset B_R:=\{x\in \R^3: |x|<R\}$.
 Together with the governing equation, we impose the initial conditions
\be\label{IBCs}
U(x,0)=0,\quad \partial_tU(x,0)=0,\quad & x\in \R^3\backslash{\overline{D}},
\en
 and the traction-free boundary condition on $\partial D$:
\be\label{stress1}
 \mathcal TU(x,t)=0,\quad& (x,t)\in \partial D\times \R^+,
\en
where $\mathcal T U$ is the stress boundary condition defined by \eqref{stress} (see Section 2). In this paper we consider the inverse problem  of determining the source term $F$ from knowledge of $U$ on the surface $\partial B_R=\{x\in \R^3: |x|=R\}$ with $R>0$ sufficiently large. According to \cite[Remark 4.5]{BHKY} there is an obstruction for the recovery of general time-dependent source terms $F$. Facing this obstruction we consider this problem for some specific type of source terms.
\subsection{Motivations}

We recall that the Lam\'e system \eqref{llame1}-\eqref{lame2} is frequently used for the study of linear elasticity and imaging problems. In this context our inverse problems can be seen as the recovery of an external force  provided by the source term $F$.
For instance, the recovery of  elastodynamics source
term $F$ corresponding to the product of a spatial function $g$ and a temporal function $f$  can be regarded as an approximation of the elastic pulse and are commonly used in modeling vibration phenomena in seismology and teleseismic inversion \cite{AR,S}. This type of sources has been also considered in numerous
applications in biomedical imaging (see the references \cite{Ammari13,Ammari} and the references therein) where our inverse problem can be seen as the recovery of the information provided by the parameter under consideration.

Let us also observe that the identification of time-dependent sources (see Theorems \ref{Th2:Uniqueness} and \ref{t2}) is associated to the recovery of a moving source which can be thought as an approximation of a pulsed signal transmitted by a moving antenna (see \cite{HKLZ} and the references therein for more details).
\subsection{Known results}

Inverse source problems are a class of inverse problems which have received many interest. These problems take different forms and have many applications (environment, imaging, seismology $\cdots$). For an overview of these problems we refer  to \cite{I}. Among the different arguments considered for solving these problems we can mention the approach based on applications of Carleman estimates arising from the work of
\cite{K1981} (see also \cite{Kh,Klibanov1992}). This approach has been applied successfully to hyperbolic equations by \cite{Ya99} in order to extend his previous work \cite{Ya95} to a wider class of source terms. More precisely, in \cite{Ya99} the author considered the recovery of source terms of the form $f(x)G(x,t)$, where $G$ is known, while in \cite{Ya95} the analysis of the author is restricted to source terms of the form $\sigma(t)f(x)$, with $\sigma$ known. More recently, the approach of \cite{Ya99} has been extended by \cite{JLY} to hyperbolic equations with time-dependent second order coefficients and to less regular coefficients by \cite{YZ}. We mention also the work of \cite{CY,KSS} using similar approach for inverse source problems stated for parabolic equations and the result of \cite{SU} proved by a combination of geometrical arguments and Carleman estimates. Concerning the Lam\'e system we refer to \cite{INY1998} where a uniqueness result has been stated for the recovery of time-independent source terms by mean of suitable Carleman estimate and we mention also the work of \cite{IY2005,IY05,Isa86} dealing with related problems as well as  \cite{BY}  where an inverse source problem for
Biot's equations has been considered. We refer  also to the recent work \cite{BHKY} where the recovery of a time-independent source term appearing in the Lam\'e system in all space has been proved from measurements outside the support of the source under consideration  as well as the work of \cite{JLLY} dealing with this inverse source problem for fractional diffusion equations. 

In all the above mentioned results the authors considered the recovery of time independent source terms (in other words, the spatial component of the source term). For the recovery of a source depending only on the time variable we refer to \cite{FK} where such problems has been considered for fractional diffusion equations and for the recovery of some class of sources depending on both space and time variable appearing in a parabolic equation on the half space, we refer to \cite[Section 6.3]{I}. For hyperbolic equations, we refer to \cite{DOT,RS} where the recovery of some specific time-dependent source terms have been considered. For Lam\'e systems, \cite[Theorem 4.2]{BHKY} seems to be the only result available in the mathematical literature where such a problem has been addressed for time-dependent source terms. The result of \cite[Theorem 4.2]{BHKY} is stated with source terms depending only on the time variable. To the best of our knowledge, except the result of \cite{DOT}, dealing with the recovery of discrete in time
sources, and the result of the present paper, there is no result in the mathematical literature treating the recovery of a source term depending on both space and time variables appearing in hyperbolic equations.

\subsection{Main results}

In the present paper we consider three inverse problems related to the recovery of the source term $F$. In our first inverse problems we assume that the cavity $D\neq\emptyset$ is a domain with $\mathcal C^3$ boundary $\partial D$, with connected exterior $\R^3\backslash\overline{D}$,  and  we consider source terms of the form
\be\label{source1}F(x,t)=f(t)h(x),\quad x\in\R^3\backslash\overline{D},\ t\in(0,+\infty),\en with $f$ a real valued function and $h=(h_1, h_2, h_3)^\top: \R^3\backslash\overline{D}\rightarrow \R^3$ a vector valued function. 
Choose $R>0$ sufficiently large such that $B_R$ also contains the support of $h$ (i.e., supp$(h)\subset B_R$). We assume that  $f\in L^2(0,+\infty)$ is compactly supported and  $h\in L^2(\R^3\backslash \overline{D})^3$.
Then, the  problem (\ref{llame1})-(\ref{stress1})  admits a unique solution
$$U\in \mathcal C^1([0,+\infty);L^2(\R^3\backslash D))^3\cap \mathcal C([0,+\infty);H^1(\R^3\backslash D))^3.$$
The proof of this result can be carried out by combining the elliptic regularity properties of  $\mathcal{L}_{\lambda,\mu}$ (see e.g., \cite[Chapters 4 and 10]{Mclean} and \cite[Chapter 5]{Hsiao}) with \cite[Theorem 8.1, Chapter 3]{LM1} and \cite[Theorem 8.2, Chapter 3]{LM1} (see also the beginning of Sections 4.1 and 4.2 for more details).
Our first inverse problem in the exterior of the cavity  can be stated as follows.

\textbf{Inverse Problem 1} (IP1): Assume that  $f$, $D$ are both known in advance. Determine the spatially dependent function $h$ from the radiated field $U$ measured on the surface $\partial B_R\times [0,T_1) $, $T_1\in(0,+\infty]$.

Below we give a confirmative answer to the uniqueness issue for IP1 in two different cases. For source terms with low regularity we obtain 


\begin{thm}\label{THH2} Let  $f\in \mathcal C^1([0,+\infty))$ satisfy $f(0)\neq 0$, $h\in H^1(\R^3\backslash \overline{D})^3$ and $F$ takes the form \eqref{source1}. Let also $\Omega:=B_R\backslash \overline{D}$ and let $d_j$, $j=1,2$, be the Riemannian distance within $\overline{\Omega}$ induced by  the metric $g_j$, where
$$g_1[x](v,v)=\frac{\rho|v|^2}{\mu(x)},\quad g_2[x](v,v)=\frac{\rho|v|^2}{2\mu(x)+\lambda(x)},\quad x\in\overline{\Omega},\ v\in\R^3.$$
Then, for \tcr{$$T_1>2\left(\underset{j=1,2}{\max}\ \left(\underset{x\in \overline{\Omega}}{\sup}\  d_j(x,\partial B_R)\right)\right),$$}  the boundary data $\{U(x,t):\ (x,t)\in \partial B_{R}\times(0,T_1)\}$  uniquely determine $h$.
\end{thm}

By considering measurements for all time ($t\in(0,+\infty)$), we can remove the condition $f(0)\neq0$ in the following way.

\begin{thm}\label{TH1} Let  $f\in H^1_0(0,T)$, $h\in H^1(\R^3\backslash \overline{D})^3$ and  $F$ takes the form \eqref{source1}.
Then the boundary data $\{U(x,t):\ (x,t)\in \partial B_R\times\R^+\}$ uniquely determine $h$.
\end{thm}

For our last inverse problem, we consider the Lam\'e system with constant density and Lam\'e coefficients when the embedded cavity is absent ($D=\emptyset$). We assume here that $F$ takes the form
\be\label{source2}F(\tilde{x},x_3,t)=g(x_3)\,f(\tilde{x},t),\quad \tilde{x}\in\R^2,\ x_3\in\R,\ t\in(0,+\infty),\en
 where the vectorial function $f=(f_1, f_2, 0)^\top$ is compactly supported on
$\tilde{B}_R\times [0, T)$ and the scalar function $g$ is supported in $(-R, R)$ for some $R>0$. Here $\tilde{x}=(x_1, x_2)\in \R^2$ for $x=(x_1,x_2,x_3)\in \R^3$ and $\tilde{B}_R$ denotes the set $\tilde{B}_R:=\{\tilde{x}\in\R^2:\ |\tilde{x}|<R\}$.
Then our last inverse problem can be stated as follows.\\
\textbf{Inverse Problem 2} (IP2): Assume that  $g$ is known in advance.  Determine the time and space dependent function  $f$ from  the radiated field $U$ measured on the surface $\Gamma\times(0,T_1) $, with $T_1>0$, $R_1>0$ sufficiently large and $\Gamma\subset\partial B_{R_1}$  an  open set with positive Lebesgue measurement.\\

In this paper we give a positive answer to (IP2) both in terms of uniqueness and stability. Our uniqueness result can be stated as follows.

\begin{thm}\label{Th2:Uniqueness} Assume $D=\emptyset$ and $\rho,\lambda,\mu$ are all constants in $\R^3$.
Assume that $F$ takes the form \eqref{source2} with $f=(f_1,f_2,0)^\top\in H^1(0,T; L^2(\R^2))^3$, $g\in L^2(-R, R)$ is non-uniformly vanishing and
 $$f( \tilde{x},0)= 0,\quad\tilde{x}\in\R^2.$$  Let $R_1>\sqrt{2}R$, $T_1>T+\frac{2R_1\sqrt{\rho}}{\sqrt{\mu}}$ and let $\Gamma\subset\partial B_{R_1}$ be an arbitrary open set with positive Lebesgue measurement. Then the source $f$ can be uniquely determined by the data $U(x, t)$ measured on $\Gamma\times (0, T_1)$.
\end{thm}

By assuming that $\Gamma=\partial B_{R_1}$, we can extend this uniqueness result to a log-type stability estimate taking the form.
For this purpose, we need a priori information on the regularity and upper bound of the source terms $f$ and $g$.

\begin{thm}\label{t2} Let  $R_1>\sqrt{2}R$, \tcr{$T_1>T+\frac{2R_1\sqrt{\rho}}{\sqrt{\mu}}$},  $\rho,\lambda,\mu$ be constant and assume that $D=\emptyset$,  $f\in H^3(\R^2\times\R)^3\cap H^4(0,T;L^2(\R^2))^3$ satisfies $$f(\tilde{x},0)=\partial_t f(\tilde{x},0)=\partial_t^2 f(\tilde{x},0)=\partial_t^3 f(\tilde{x},0)=0,\quad \tilde{x}\in\R^2.$$ Assume also that $g$ is non-uniformly vanishing with a constant  sign \emph{($g\geq0$ or $g\leq0$)} and that there exists $M>0$ such that
\begin{equation}\label{t2a} \norm{f}_{H^3(\R^2\times\R)^3}+\norm{f}_{H^4(0,T;L^2(\R^2))^3}\leq M.\end{equation}
Then, there exists $C>0$ depending on $M$, $R_1$, $\rho$, $\lambda$, $\mu$, $T_1$, $\norm{g}_{L^1(\R)}$ such that
\begin{equation}\label{t2b}
||f||_{L^2((0,T)\times\tilde{B}_{R})}\leq  C\left(\norm{U}_{H^3(0,T_1;H^{3/2}(\partial B_{R_1}))^3}+\left|\ln\left(\norm{U}_{H^3(0,T_1;H^{3/2}(\partial B_{R_1}))^3}\right)\right|^{-1}\right).\end{equation}
\end{thm}

\subsection{Comments about our results}

Let us first remark that to the best of our knowledge Theorems  \ref{THH2} and \ref{TH1} are the first results of recovery of source terms  stated for the Lam\'e system outside a cavity with variable coefficients. Indeed, it seems that all other known results have been stated on a bounded domain (e.g. \cite{INY1998}) or in the full space $\R^3$ (e.g. \cite{BHKY}). We emphasize that Theorems \ref{THH2} and \ref{TH1} are valid even if the embedded cavity is unknown. In fact, the unique determination of the embedded cavity can be proved following Isakov's arguments \cite[Theorem 5.1]{Isakov} by applying the unique continuation results of \cite{EINT, ET}; see also the proof of Theorem \ref{THH2}. We refer also to \cite{Isakov08} for the determination of other impenetrable scatterers for the wave equation with a single measurement data. The main purpose of this paper is concerned with the identification of elastic sources in an unbounded domain. 

Let us observe that in  Theorem  \ref{THH2},  we manage to restrict our measurements to a finite interval of time. However, like in \cite{INY1998,SU,Ya99,YZ}, we need to impose the additional condition $f(0)\neq0$ for a source term $F$ of the form  \eqref{source1}. In contrast to Theorem  \ref{THH2} and results using Carleman estimates like \cite{INY1998,SU,Ya99,YZ}, in Theorem  \ref{TH1} we state our result without assuming that the source under consideration is non-vanishing at $t=0$. For a source term $F$ of the form  \eqref{source1}, such assumption will be equivalent to the requirement that $f(0)\neq0$. From the practical point of view, this means that the results of  \cite{INY1998,SU,Ya99,YZ}, as well as Theorem  \ref{TH1}, can only be applied to the determination of a source term associated with a phenomenon which has appeared before the beginning of the measurement. This restriction excludes applications where one wants to determine a phenomenon  with measurements that start before its appearance. By removing this restriction in Theorem \ref{TH1} we make our result more suitable for applications in that context. The approach of Theorem \ref{THH2}, \ref{TH1} consist in  transforming our problem into the recovery of initial condition. Then, applying some results of unique continuation and global Holmgren theorem borrowed from \cite{EINT,ET} we complete the proof of Theorem \ref{THH2}, \ref{TH1}.

To the best of our knowledge,  even for a bounded domain, Theorems \ref{Th2:Uniqueness} and \ref{t2} seem to be the first results of unique and stable recovery of some general source term depending on both time and space variables appearing in a hyperbolic equation. Indeed, it seems that only results dealing with recovery of source terms depending only on the time variable (see \cite{BHKY,RS}) or space variable (see \cite{BHKY,INY1998,SU,Ya99,YZ}) are available in the mathematical literature with the exception of \cite{DOT} where the recovery of discrete in time sources  has been considered. Therefore the results of Theorems \ref{Th2:Uniqueness} and \ref{t2} are not only new for the Lam\'e system but also more general for hyperbolic equations. We mention also that the stability result of Theorem  \ref{t2} requires a result of stability in the unique continuation already considered by \cite{Ba,CK,Ki1} for the recovery of time-dependent coefficients. Note also that, in contrast of Theorems \ref{TH1} and  \ref{THH2}, thanks to the strong Huygens principle we can state Theorems \ref{Th2:Uniqueness} and \ref{t2} at finite time.

In Corollary \ref{cor1}  we prove that the results  of  Theorem \ref{TH1}  can be reformulated in terms of partial recovery of the source term  from measurement on a subdomain where the source term or the initial data are known. This situation may for instance occur in several applications where the source under consideration has large support and the data considered in Theorem \ref{TH1}  is not accessible. What we prove in Corollary \ref{cor1}  is that even in such context one can expect recovery of partial information of the source term under consideration by measurements located on some subdomain where the source is known.

Both Theorems \ref{THH2}, \ref{TH1} and Corollary \ref{cor1} remain valid if the cavity $D$ is absent
 or if $D$ is a rigid elastic body (i.e., $U$ vanishes on $\partial D$). All the results of this paper can be applied to the wave equation. Actually, the proof for the wave equation will be easier in several aspects and the particular treatment for the Lam\'e system leads to some  difficulties inherent to this type of systems (see for instance the proof of Theorems \ref{Th2:Uniqueness} and \ref{t2}).
\subsection{Outline}

This paper is organized as follows. In Section 2 we study the inverse problem (IP1). More precisely, we prove Theorems \ref{THH2} and \ref{TH1} as well as the Corollary \ref{cor1}. In Section 3 we treat the inverse problem (IP2). We start with the uniqueness result stated in Theorem \ref{Th2:Uniqueness}. Then, we extend this result by proving the  stability estimate stated in Theorem \ref{t2}. We give also some results related to solutions of the problem \eqref{llame1}-\eqref{stress1} in the appendix.

\section{Inverse source problem with traction-free boundary condition}
This section is devoted to the proofs of Theorems \ref{THH2} and \ref{TH1} for our inverse problem (IP1).
More precisely, we consider the radiation of an elastic source in an inhomogeneous medium in the exterior of a cavity $D$ (see Figure 1):
\be\label{lame}
&&\rho \partial_{tt}U(x,t)=\mathcal L_{\lambda,\mu} U(x,t)+f(t)\,g(x), \quad x=(x_1,x_2,x_3)\in \R^3\backslash{\overline{D}},\; t\in(0,+\infty),
\en
where $\mathcal{L}_{\lambda,\mu}U$ stands for the Lam\'e operator given by \eqref{lame2}.

\begin{figure}[!ht]
  \centering
  \includegraphics[width=0.4\textwidth]{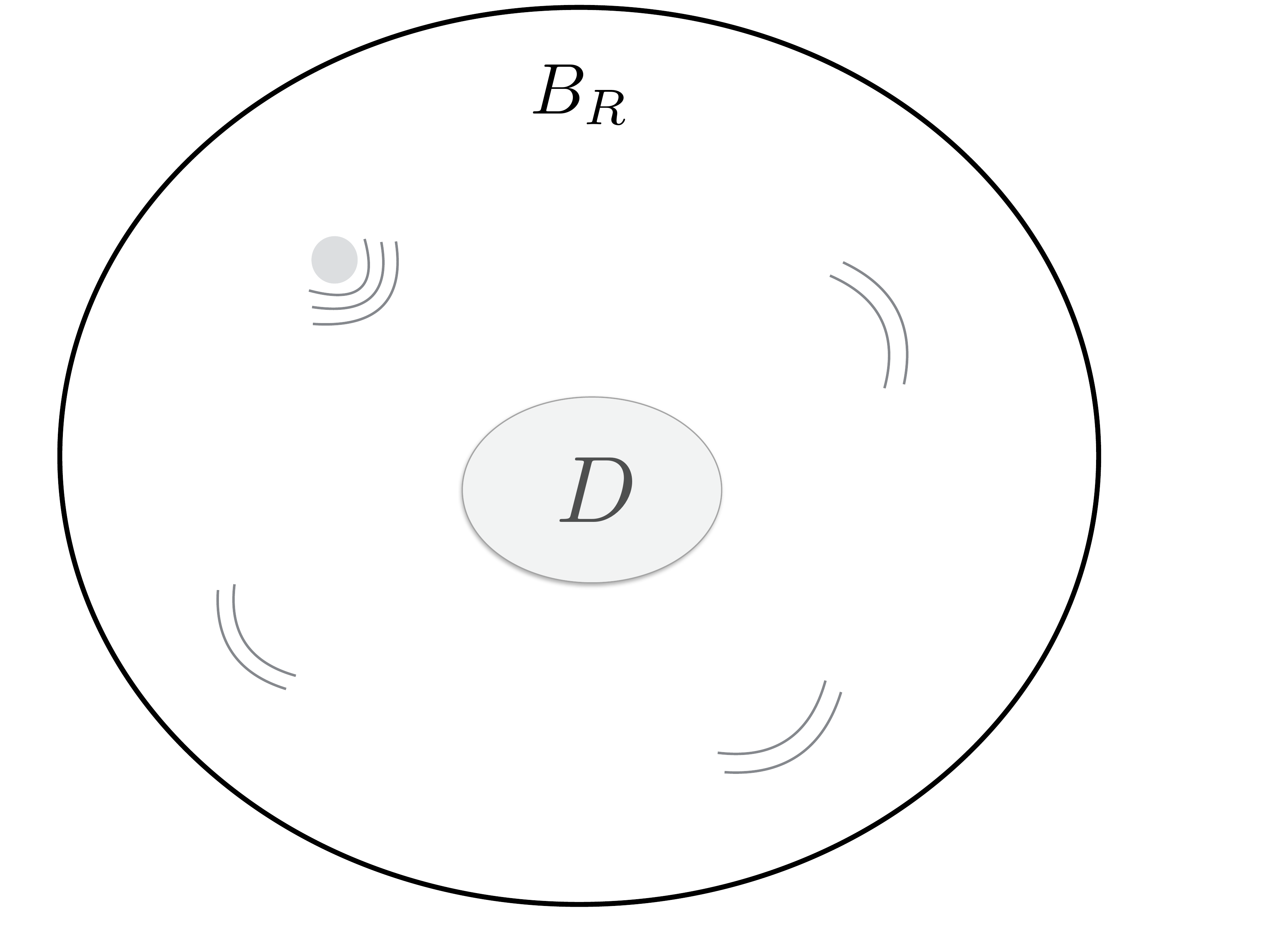}
  \caption{Radiation of a source in an inhomogeneous isotropic elastic medium in the exterior of a cavity. Suppose that the cavity $D$ is known. The inverse problem is to determine the source term from the data measured on $\partial B_R=\{x\in \R^3: |x|=R\}$.  }
  \label{fig1}
\end{figure}

Together with the governing equation (\ref{lame}), we fix the initial conditions at $t=0$:
\be\label{II}
U(x,0)=0,\quad \partial_tU(x,0)=0,\quad & x\in \R^3\backslash{\overline{D}},
\en
 and the traction-free boundary condition $\mathcal T U$ on $\partial D$ given by
\be\label{stress}
 \mathcal T U:=\sigma(U)\nu=0\quad& \mbox{on}\quad \partial D\times \R^+,
\en
where $\nu=(\nu_1,\nu_2,\nu_2)$ stands for the unit normal direction pointing into the exterior of $D$ and the stress tensor $\sigma (U)$ is given by
\be\label{E}
\sigma(U):=\lambda\; \Div U\,\textbf{I}_{3}+2\mu E(U),\quad E(U):=\frac{1}{2} ((\nabla U)+(\nabla U)^T).
\en
Note that $\textbf{I}_{3}$ means the 3-by-3 unit matrix and that the
conormal derivative $\sigma(U)\nu$ corresponds to
the \emph{stress vector} or \emph{surface traction} on $\partial D$.
With these notation the Lam\'e operator (\ref{lame2}) can be written as $\mathcal{L}_{\lambda,\mu} U=\Div \sigma(U)$.

We suppose that $D\subset\R^3$ is a bounded domain with $\mathcal{C}^3$-smooth boundary $\partial D$ and with connected exterior $\R^3\backslash\overline{D}$.
 If the cavity $D$ is absent (i.e., $D=\emptyset$) and the background medium is homogeneous and isotropic, it was shown in \cite{BHKY} via strong Huygens principle and Fourier transform that the boundary data of Theorem \ref{TH1} can be used to uniquely determine $g$. According to \cite{Ka}, in the context of Theorem \ref{TH1}, the strong Huygens principle is not  valid and we can not even expect integrable local energy decay. For this purpose, we   use a different approach based on application of Laplace transform for Theorem \ref{TH1} and unique continuation properties for Theorem \ref{THH2}.

Let us first consider the proof of Theorem \ref{THH2} stated with non-vanishing sources at $t=0$.

\textbf{Proof of Theorem \ref{THH2}.}
 Assuming $U(x,t)=0$ for $|x|=R$ and $t\in (0,T_1)$, we need to prove that $h\equiv 0$. Since ${\rm Supp}(h)\subset B_R$, the wave field $U$ fulfills the homogeneous initial and boundary conditions of the Lam\'e system in the exterior of $B_R$:
\begin{equation}\label{dd}\left\{\begin{array}{lll}
\rho\partial_{tt}U(x,t)-\mathcal{L}_{\lambda,\mu} U(x,t)=0\quad&\mbox{in}\quad \R^3\backslash\overline{B}_R\times (0,T_1),\\
U(x,0)=\partial_tU(x,0)=0\quad  &\mbox{in}\quad \R^3\backslash\overline{B}_R, \\
U(x,t)=0\quad &\mbox{on}\quad \partial B_R\times \R^+.
\end{array}\right.\end{equation}
Applying the elliptic regularity properties of  $\mathcal{L}_{\lambda,\mu}$ (see e.g., \cite[Chapters 4 and 10]{Mclean} and \cite[Chapter 5]{Hsiao}) and the results of \cite[Theorem 8.1, Chapter 3]{LM1}, \cite[Theorem 8.2, Chapter 3]{LM1} (see also the beginning of Sections 4.1 and 4.2 for more details), one can prove the unique solvability of the initial boundary value problem \eqref{dd}. Consequently, we deduce that
$U\equiv 0$ in $(\R^3\backslash\overline{B}_R)\times [0,T_1)$.

Let us now consider the initial boundary value problem
\begin{equation}\label{eeq1}\left\{\begin{array}{ll}\rho\partial_t^2V+\mathcal{L}_{\lambda,\mu} V=0,\quad &(x,t)\in(\R^3\backslash \overline{D})\times(0,+\infty),\\  V(\cdot,0)=0,\quad \partial_tV(\cdot,0)=h,\quad &\textrm{in}\ \R^3\backslash \overline{D},\\ \mathcal T V(x,t)=0,\quad& (x,t)\in\partial D\times(0,+\infty).\end{array}\right.\end{equation}
Analogous to the boundary value problem (\ref{dd}), one can prove that  this exterior problem admits a unique solution $V\in \mathcal C([0,T];H^2(\R^3\backslash D)^3)\cap \mathcal C^1([0,T];H^1(\R^3\backslash D)^3)$.  Moreover, one can easily check that the solution $U$ to (\ref{lame}) is connected with $V$ via (which is well-known as Duhamel's principle)
\begin{equation}\label{l4a}U(x,t)=\int_0^tf(t-s)V(x,s)ds,\quad t\in(0,+\infty),\;x\in \R^3\backslash\overline{D},\end{equation}
Combining  \eqref{l4a} with the fact that $U\equiv 0$ in $(\R^3\backslash\overline{B}_R)\times [0,T_1)$, we deduce that
$$\int_0^tf(t-s)V(\cdot,s)_{|\R^3\backslash B_R}ds=0,\quad t\in[0,T_1).$$
Using the fact that $f\in\mathcal C^1([0,T])$, we can differentiate this expression with respect to $t$ in order to get
$$
f(0)V(\cdot,t)_{|\R^3\backslash B_R}+\int_0^tf'(t-s)V(\cdot,s)_{|\R^3\backslash B_R}ds=0,\quad t\in[0,T_1).$$
Combining this with the fact that $f(0)\neq 0$, we obtain
$$ \norm{V(\cdot,t)}_{L^2(\R^3\backslash B_R)}\leq \frac{\norm{f}_{\mathcal{C}^1(0,T_1)}}{|f(0)|}\left(\int_0^t\norm{V(\cdot,s)}_{L^2(\R^3\backslash B_R)}ds\right),\quad t\in[0,T_1).$$
Therefore, applying the Gronwall inequality, we deduce that
\begin{equation}\label{l4b}V(x,t)=0,\quad t\in[0,T_1),\quad x\in \R^3\backslash B_R.\end{equation}
From this, we deduce that
$$V(x,t)=\mathcal T V(x,t)=0,\quad x\in \partial B_R,\ t\in[0,T_1).$$
Combining this with the fact that \tcr{$T_1>2\left(\underset{j=1,2}{\max}\ \left(\underset{x\in \overline{\Omega}}{\sup}\  d_j(x,\partial B_R)\right)\right)$}, we deduce that there exist
 $$\tcr{T_2>\left(\underset{j=1,2}{\max}\ \left(\underset{x\in \overline{\Omega}}{\sup}\  d_j(x,\partial B_R)\right)\right)},\quad\epsilon>0,$$
such that
\begin{equation}\label{l4c}V(x,t)=\mathcal T V(x,t)=0,\quad x\in \partial B_R,\ t\in[0,2T_2+4\epsilon].\end{equation}

Combining the unique continuation result of \cite[Theorem 5.5]{EINT} with the global Holmgren theorem stated in \cite[Theorem 1.2]{ET} and repeating arguments similar to \cite[Theorem 3.2]{ET} (see also the properties of Lam\'e system recalled at the beginning of  Section 3.2 of \cite{ET} as well as  \cite[Remark 3.5]{ET}), we deduce that \eqref{l4c} implies
$$V(x,t)=0,\quad x\in B_R\backslash \overline{D},\ t\in (T_2,T_2+2\epsilon).$$
Combining this with \eqref{l4b}, we get
$$V(x,t)=0,\quad x\in \R^3\backslash \overline{D},\ t\in (T_2,T_2+2\epsilon)$$
and differentiating with respect to $t$, we get
$$V(x,T_2+\epsilon)=\partial_tV(x,T_2+\epsilon)=0,\quad x\in \R^3\backslash \overline{D}.$$
Therefore, $V$ restricted to $(\R^3\backslash \overline{D})\times (0,T_2+\epsilon)$ solves the initial boundary value problem
$$\left\{\begin{array}{ll}\rho\partial_t^2V+\mathcal{L}_{\lambda,\mu} V=0,\quad &(x,t)\in(\R^3\backslash \overline{D})\times(0,T_2+\epsilon),\\  V(\cdot,T_2+\epsilon)=0,\quad \partial_tV(\cdot,T_2+\epsilon)=0,\quad &\textrm{in}\ \R^3\backslash \overline{D},\\ \mathcal T V(x,t)=0,\quad& (x,t)\in\partial D\times(0,T_2+\epsilon).\end{array}\right.$$
The uniqueness of the solution of this problem implies that
$$V(x,t)=0,\quad x\in \R^3\backslash \overline{D},\ t\in(0,T_2+\epsilon)$$
from which we deduce that $h\equiv 0$.\qed

Now let us consider Theorem \ref{TH1}, where we allow $f(0)=0$ but we make measurements for all time.

\textbf{Proof of Theorem \ref{TH1}.} Assuming $U(x,t)=0$ for $|x|=R$ and $t\in \R^+$, we need to prove that $h\equiv 0$. Repeating the arguments used at the beginning of Theorem \ref{THH2}, one can check that
$$U(x,t)=0,\quad t\in[0,+\infty),\quad x\in \R^3\backslash B_R.$$
Then, repeating the arguments used in  Theorem \ref{THH2} with some minor modifications, we can prove that
\begin{equation}\label{TH1d}\int_0^tf(t-s)V(\cdot,s)_{|\R^3\backslash B_R}ds=0,\quad t\in[0,+\infty),\end{equation}
with $V$ the solution of \eqref{eeq1}. Since $f(t)\equiv 0$ in $t<0$, the identity (\ref{TH1d}) can be rewritten as
\begin{equation}\label{THd}
f(t)*V(\cdot, t)_{|\R^3\backslash B_R}=\int_0^\infty f(t-s)V(\cdot,s)_{|\R^3\backslash B_R}ds=0,\quad t\in[0,+\infty)\end{equation}
 where the operator $*$ denotes the convolution.
For $R_1>R$, we fix  $\Omega_1:=B_{R_1}\backslash \overline{B_R}$. Using standard idea for deriving energy
estimates, one can prove that $t\mapsto\norm{V(\cdot,t)}_{H^1(\Omega_1)}$ has a long time  behavior  which is at most of polynomial type
(see Proposition \ref{p4} in the Appendix ).
This allows us to define the Laplace transform of $t\mapsto V(\cdot,t)_{|\Omega_1}$ with respect to the time variable as following:
\ben
\hat{V}(x,\tau):=\int_{\R} V(x,t)\, e^{-\tau\,t}\,dt,\quad \tau>0,\quad x\in \Omega_1,
\enn
and $z\mapsto\hat{V}(\cdot, z)_{|\Omega_1}$ is an holomorphic function on $\mathbb C_+:=\{z\in\mathbb C:\ \re z>0\}$ taking values in $H^1(\Omega_1)^3$. Therefore applying the laplace transform to both sides of \eqref{THd}, we get
\begin{equation}\label{TH1e}\hat{f}(\tau)\hat{V}(x,\tau)=0,\quad x\in\Omega_1,\ \tau>0.\end{equation}
Using the fact that $f\in L^1(\R_+)$ is supported in $[0,T]$ and it does not vanish identically, we deduce that the function $\hat{f}$ is holomorphic in $ \C$ and not identically zero. Thus,  there exists an interval $I\subset (0,+\infty)$ such that $|\hat{f}(\tau)|>0$ for $\tau\in I$. Combining this with \eqref{TH1e}, we deduce that
$$\hat{V}(x,\tau)=0,\quad x\in\Omega_1,\ \tau\in I$$
and using the fact that $z\mapsto\hat{V}(\cdot, z)_{|\Omega_1}$ is an holomorphic function on $\mathbb C_+:=\{z\in\mathbb C:\ \re z>0\}$, we deduce that
$$\hat{V}(x,\tau)=0,\quad x\in\Omega_1,\ \tau>0.$$
Then,  the injectivity of the Laplace transform, implies
$$V(x,t)=0,\quad x\in\Omega_1,\ t>0.$$
Combining this with the arguments used at the end of Theorem \ref{THH2}, we deduce that $h\equiv0$.
\qed

We remark that surface data are utilized in the proof of
 Theorem \ref{TH1}. As a corollary, we prove that interior volume observations can also be used to extract partial information of the spatial source term.
Below we consider again the  problem (\ref{lame})-(\ref{stress}), with $f, g$ being given as in Theorem \ref{TH1}.

 \begin{cor}\label{cor1} Suppose that $f$ is given and let $\Omega$ be a $\mathcal C^3$-smooth connected open set of $\R^3$  satisfying
 $D\subset \Omega\subset B_R$ for some $R>0$. Let $\omega$ be an open set of $\R^3$.
Then the wave fields $U$ measured on the volume
 $\omega\times\R^+$ and on the surface  $\partial\Omega\times\R^+$ (see figure \ref{fig2}) uniquely determine $h|_{\Omega}$.
\end{cor}
\begin{figure}[!ht]
  \centering
  \includegraphics[width=0.5\textwidth]{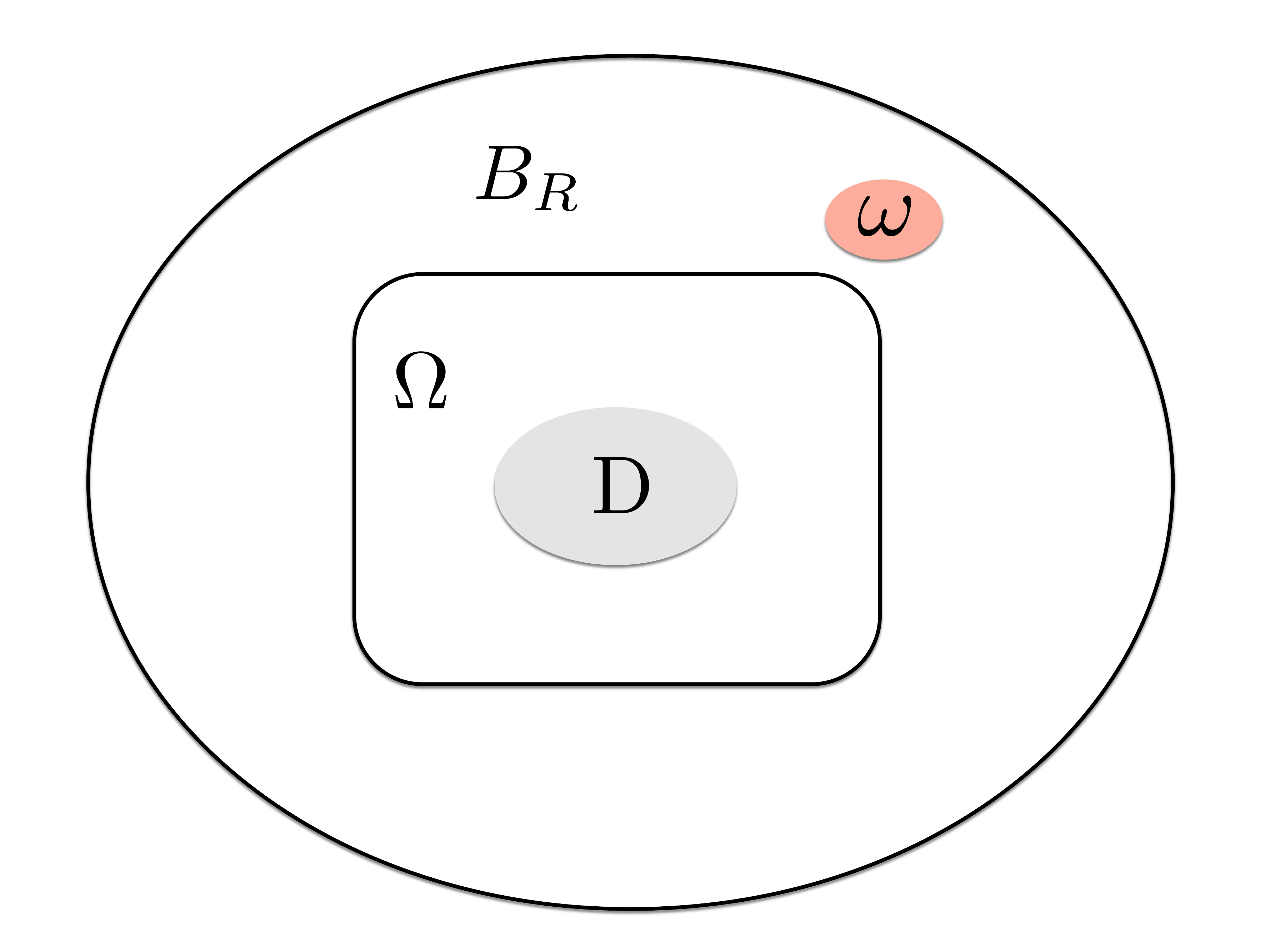}
  \caption{Suppose that  the data are collected on $\omega$ and on $\partial\Omega$. The inverse problem is to determine the value of $g$ on $\Omega$. }
  \label{fig2}
\end{figure}

\begin{proof}  We need to prove that the condition $U(x,t)=0$ for $(x,t)\in(\omega\times\R^+)\cup (\partial\Omega\times\R^+)$ implies that $h|_{\Omega}\equiv 0$. Repeating the arguments used in Theorem \ref{TH1}, for $V$ the solution of \eqref{eeq1}, we have
$$V(x,t)=0,\quad x\in\omega\cup\partial\Omega ,\ t>0.$$
In a similar way to Theorem \ref{THH2}, combing  the unique continuation result of \cite[Theorem 5.5]{EINT} with the global Holmgren theorem stated in \cite[Theorem 1.2]{ET}, we deduce that the condition
$$V(x,t)=0,\quad x\in\omega ,\ t>0$$
implies that there exists $T_3>0$ such that
$$V(x,T_3)=\partial_tV(x,T_3)=0,\quad x\in\Omega.$$
Combining this with the fact that
$$V(x,t)=0,\quad x\in\partial\Omega ,\ t>0,$$
we deduce that the restriction of $V$ to  $\Omega\times\R^+$ solves the problem
\begin{equation}\label{c1b}\left\{\begin{array}{ll}\rho\partial_t^2V+\mathcal{L}_{\lambda,\mu} V=0,\quad &(x,t)\in\Omega\times(0,T_3),\\  V(\cdot,T_3)=0,\quad \partial_tV(\cdot,T_3)=0,\quad &\textrm{in}\ \Omega,\\V(x,t)=0,\quad& (x,t)\in\partial \Omega\times(0,T_3)\\  \mathcal TV(x,t)=0,\quad& (x,t)\in\partial D\times(0,T_3).\end{array}\right.\end{equation}
Then, the uniqueness of the solution of this problem implies that
$$V(x,t)=0,\quad x\in\Omega,\ t\in[0,T_3].$$
In particular, we obtain $h(x)=\partial_tV(x,t)|_{t=0}=0$ in $\Omega$.
\end{proof}

\begin{rem} Corollary \ref{cor1} shows that, the volume observation data on $\omega$ and the surface measurements on $\partial\Omega$ unique determine the source term $h$ on $\Omega$. This gives partial information of $h$ only. However,
in the special case that ${\rm Supp}(h)\subset \Omega$ $($for instance, $\Omega=B_R$$)$, one may deduce from Corollary \ref{cor1} that
 $h$ can be uniquely determined by the data of $U$ on $\omega\times\R^+$.\end{rem}

\begin{rem} Assuming that $f(0)\neq0$, in a similar way to Theorem \ref{THH2}, we can restrict the measurements in Corollary \ref{cor1} to a finite time depending on the coefficients $\rho,\lambda,\mu$ and the domain $\Omega$, $\omega$.\end{rem}

\section{Determination of the source term $g(x_3) f(\tilde{x},t) $}
In the previous section, we established uniqueness of recovering a spatial source term in an inhomogeneous background medium with or without embedded obstacles.
However, the dependance of the source term on time and spatial variables
are completely separated. The counterexamples constructed in \cite{BHKY}
show that it is impossible to recover general source terms of the form $F(x,t)$
from the boundary observation on $\partial B_R\times(0,\infty)$.
This implies that a priori information on the source term is always necessary in proving uniqueness.
In this section we restrict our discussions to the inverse problem (IP2) for alternative source terms of
the form $g(x_3)f( \tilde{x},t) $, where the vectorial function $f=(f_1, f_2, 0)$ is compactly supported on
$[0, T)\times \tilde{B}_R$ and the scalar function $g$ is supported in $(-R, R)$ for some $R>0$. We recall that here $\tilde{x}=(x_1, x_2)\in \R^2$ for $x=(x_1,x_2,x_3)\in \R^3$, and $\tilde{B}_R:=\{\tilde{x}: |\tilde{x}|\leq R\}$.

For simplicity, we assume in this section that $D=\emptyset$ and the background medium is homogeneous with constant
Lam\'e coefficients $\lambda$, $\mu$ and a constant density function $\rho$. Below we shall consider the initial value problem
\be\label{Lame2}\left\{\begin{split}
&\rho \partial_{tt}U(x,t)=\mathcal{L}_{\lambda,\mu} U(x,t)+g(x_3)\,f(\tilde{x},t), && x\in \R^3,\; t>0,\\
&U(x,0)=\partial_t\,U(x,0)=0, && x\in \R^3.
\end{split}\right.
\en
The function $g(x_3)\,f(\tilde{x},t)$ can be used to
model source terms which mainly radiate over the $ox_1x_2$-plane and $g(x_3)$ can be regarded as an approximation of the delta function $\delta(x_3)$ in the $x_3$-direction.
 Suppose that the function $g$ is known in advance.
Our inverse problem in this section is concerned with the recovery of $f$ from $U(x, t)$ measured on $\Gamma\times (0, T_1)$ for some $T_1>0$,  $R_1>\sqrt{2}R$ and $\Gamma$ an open subset of $\partial B_{R_1}$.
The proofs of the uniqueness and stability results (Theorems \ref{Th2:Uniqueness} and \ref{t2}) will be presented in the subsequent two subsections.

Recall that by Lemma \ref{l1} in the appendix, the boundary value problem (\ref{Lame2}) admits a unique solution in
$\mathcal C^2([0,+\infty);L^2(\R^3))^3\cap \mathcal C([0,+\infty);H^2(\R^3))^3$
under the assumption $f(\tilde{x},0)=0$.
Below we prove the uniqueness  with partial boundary data measured over a finite time. 
\subsection{Proof of Theorem \ref{Th2:Uniqueness}}
By the strong Huygens principle, fixing $\epsilon >0$, it holds that $U(x, t)= 0$ for all $|x|<R_1+\epsilon$ and $t>T_1+2\epsilon$ (see e.g. \cite{BHKY}). Then, applying the Fourier transform in time to $U$, with $U_{|\partial B_{R_1}\times(-\infty,0]}=0$, gives
\be\label{Lame3}
\mathcal{L}_{\lambda,\mu} \hat{U}(x,\omega)+\omega^2\rho \hat{U}(x,\omega)=-g(x_3)\,\hat{f}(\tilde{x},\omega), \quad &x\in \R^3,\; \omega\in \R,
\en
where
$$\hat{U}(x, \omega):=\int_\R U(x,t)e^{-i\omega t}dt,\qquad \omega\in\R$$
 satisfies the Kupradze radiation condition as $|x|\rightarrow\infty$ (see \cite{BHKY,Ku}) for any fixed $\omega\in \R$. Here $\hat{f}(\tilde{x},\omega)$ denotes the Fourier transform of $f( \tilde{x},t)$ with respect to the time variable.
Evidently, we have the boundary condition $\hat{U}(x, \omega)=0$, $x\in\Gamma$, $\omega\in\R$.
Since, for all $\omega\in\R$, the support of the function $x\mapsto\hat{f}( \tilde{x},\omega) g(x_3))$ is  contained into $B_{R_1}$,
by elliptic interior regularity, we deduce that $x\mapsto\hat{U}(x, \omega)$ is analytic with respect to the spatial variable $x$ in a neighborhood of $\partial B_R$. By analyticity of both the surface $\partial B_{R_1}$ and the function $\hat{U}(\cdot, \omega)$, we get the vanishing of $\hat{U}(x, \omega)$ on the whole boundary $\partial B_{R_1}$ for any $\omega\in\R$.
In view of the uniqueness to the Dirichlet boundary value problem in the unbounded domain $|x|>R_1$ (see e.g., \cite{BHSY2018}), we get
\ben
\hat{U}(x, \omega)= 0,\quad\quad |x|>R_1,\quad \omega\in \R.
\enn
Consequently, we have $T\hat{U}(x,\omega)=0$ on $\partial B_{R_1}$. Since the source term $f=(f_1, f_2, 0)^\top$ is compactly supported on $\tilde{B}_R$,  by Hodge decomposition the function
$\hat{f}$ can be spatially decomposed into the form
\be\label{decF}
\hat{f}(\tilde{x},\omega)=\begin{pmatrix}
\nabla_{\tilde{x}}\; \hat{f}_p(\tilde{x},\omega) \\ 0
\end{pmatrix}
+ \begin{pmatrix}\nabla_{\tilde{x}}^\perp\; \hat{f}_s(\tilde{x},\omega) \\ 0
\end{pmatrix},
\en
where $\hat{f}_p(\cdot,\omega)$ and $\hat{f}_s( \cdot, \omega)$ are scalar functions compactly supported on $\tilde{B}_R$ as well. Here $\nabla_{\tilde{x}}=(\partial_1, \partial_2)^\top$, $\nabla_{\tilde{x}}^\perp=(-\partial_2, \partial_1)^\top$. For $\xi=(\xi_1,\xi_2)\in \R^2$ satisfying
$$|\xi|> k_s>k_p,\qquad \; k_p^2:=\frac{\omega^2\rho}{\lambda+2\mu},\quad k_s^2:=\frac{\omega^2\rho}{\mu},$$
we introduce the test functions
\ben
&&V_p(x,\omega)=\begin{pmatrix}
-i\xi_1 \\ -i \xi_2 \\ \sqrt{|\xi|^2-k_p^2}
\end{pmatrix} e^{-i\xi\cdot \tilde{x}+ \sqrt{|\xi|^2-k_p^2}\; x_3},\\
&&V_s(x,\omega)=\begin{pmatrix}
i\xi_2 \\ -i \xi_1 \\ 0
\end{pmatrix} e^{-i\xi\cdot \tilde{x}+ \sqrt{|\xi|^2-k_s^2}\; x_3}.
\enn
The numbers $k_p$ and $k_s$ denote respectively the compressional and shear wave numbers in the frequency domain.
One can easily check that $\nabla_{\tilde{x}}^\perp\cdot V_p\equiv 0$, $\nabla_{\tilde{x}}\cdot V_s\equiv 0$ in $\R^3$ and, using the fact that
$$\nabla_x\times V_p=0,\quad (\lambda+2\mu)\nabla_x\nabla_x\cdot V_p=-\omega^2\rho V_p,$$
$$\nabla_x\cdot V_s=0,\quad -\mu\nabla_x\times (\nabla_x\times V_s)=-\omega^2\rho V_s,$$
we deduce that $V_\alpha$ ($\alpha=p,s$) satisfies the homogeneous Lam\'e system in the frequency domain
\ben
\mathcal{L}_{\lambda,\mu} V_\alpha(x,\omega)+\omega^2\rho\; V_\alpha (x,\omega)=0\quad \mbox{in}\quad \R^3,\qquad x\in \R^3,\quad \alpha=p,s,
\enn
for any fixed $\omega\in \R$.
Now, multiplying $V_p$ to both sides of the equation (\ref{Lame3}) and applying Betti's formula, we obtain
 \ben
 &&\int_{B_{R_1}}\left( \mathcal{L}_{\lambda,\mu} \hat{U}(x,\omega)+\omega^2\rho \hat{U}(x,\omega)\right) \cdot V_p(x,\omega)\,dx \\
 &=& \int_{\partial B_{R_1}}  \mathcal T \hat{U}(x,\omega)\cdot V_p(x,\omega)-  \mathcal T  V_p(x,\omega)\cdot \hat{U}(x,\omega)\,ds(x)\\
 &=& 0,
 \enn
 where we have used the vanishing of the Cauchy data of $\hat{U}$ on $\partial B_{R_1}$.
 On the other hand, making use of (\ref{decF}) together with the relation $\nabla_{\tilde{x}}^\perp\cdot V_p\equiv 0$ yields
\ben
0&=&\int_{B_{R_1}} V_p(x,\omega)\cdot \hat{f}(x)(\tilde{x},\omega)
g(x_3)  dx\\
&=& \int_{B_{R_1}}
V_p(x,\omega)\cdot\begin{pmatrix}
\nabla_{\tilde{x}} \hat{f}_p(\tilde{x},\omega) \\ 0 \end{pmatrix} \;g(x_3) dx\\
&=&\left( \int_{\tilde{B}_R}
\begin{pmatrix} -i\xi_1 \\ -i\xi_2 \\ 0\end{pmatrix}e^{-i\xi\cdot \tilde{x}}
\cdot\begin{pmatrix}
\nabla_{\tilde{x}} \hat{f}_p(\tilde{x},\omega) \\ 0 \end{pmatrix} d\tilde{x}\right)\;
\left(\int_{-R}^R g(x_3) e^{\sqrt{|\xi|^2-k_p^2}\; x_3}\,dx_3\right)\\
&=&|\xi|^2\, \left(\int_{\tilde{B}_R} e^{-i\xi\cdot \tilde{x}} \hat{f}_p(\tilde{x},\omega)\,d\tilde{x}\right)\left(\int_{-R}^R g(x_3) e^{\sqrt{|\xi|^2-k_p^2}\; x_3}\,dx_3\right)
\enn
for all  $\omega\in \R$ and $\xi\in \R^2$ satisfying $|\xi|>k_s$. Since $g$ is compactly supported and lies in the space $L^1((-R,R))$, the function
$$\mathbb C\ni z\mapsto \int_{-R}^R g(x_3) e^{z\; x_3}\,dx_3$$
is holomorphic in $\mathbb C$. Then, using the fact that $g$ is not uniformly vanishing, for every $\omega\in\mathbb R$, we can find an open and not-empty interval $I_\omega\subset (k_s,+\infty)$ such that
$$\int_{-R}^R g(x_3) e^{\sqrt{|\xi|^2-k_p^2}\; x_3}\,dx_3\neq0,\quad \xi\in\R^2,\ |\xi|\in I_\omega.$$
 Hence, for every $\omega\in\R$, we have
\begin{equation}\label{ana}
\int_{\tilde{B}_R} e^{-i\xi\cdot \tilde{x}} \hat{f}_p(\tilde{x},\omega)\,d\tilde{x}=0\quad \mbox{for all}\quad \xi\in\R^2,\ |\xi|\in I_\omega.
\end{equation}
This implies that, for $\omega\in\R$ and for $\hat{f}_p(\cdot,\omega):\tilde{x}\mapsto \hat{f}_p(\tilde{x},\omega)$, the Fourier transform $\mathcal F_{\tilde{x}}[\hat{f}_p](\xi)$ of $\hat{f}_p(\cdot,\omega)$ with respect to $\tilde{x}\in\R^2$ vanishes for $\xi\in\{\eta\in\R^2: |\eta|\in I_\omega\}$. On the other hand, since, for all $\omega\in\R$, $\hat{f}_p(\cdot,\omega)$ is supported in $\tilde{B}_R$, the function
$$\xi\mapsto \int_{\R^3} e^{-i\xi\cdot \tilde{x}} \hat{f}_p(\tilde{x},\omega)\,d\tilde{x}=\int_{\tilde{B}_R} e^{-i\xi\cdot \tilde{x}} \hat{f}_p(\tilde{x},\omega)\,d\tilde{x}$$
is real analytic with respect to $\xi\in\R^2$. Then, using the fact that the set $\{\xi\in\R^2:\ |\xi|\in I_\omega\}$ is an open subset of $\R^2$, it follows from \eqref{ana} that
$$\int_{\tilde{B}_R} e^{-i\xi\cdot \tilde{x}} \hat{f}_p(\tilde{x},\omega)\,d\tilde{x}=0\quad \mbox{for all}\quad \xi\in\R^2.$$
 Applying the inverse Fourier transform in $\tilde{x}$, we get $\hat{f}_p(\cdot,\omega)=0$
 for all $\omega\in\R$. Further, applying the inverse Fourier transform in $t$ yields $f_p(\tilde{x}, t)\equiv0$ for all $\tilde{x}\in \tilde{B}_R$ and $t>0$.
 The fact that $f_s\equiv 0$ can be verified analogously by multiplying $V_s$ to both sides of (\ref{Lame3}).
 This finishes the proof of the relation $f\equiv 0$ in $\tilde{B}_R\times(0,T)$.
$\hfill\Box$

To derive a stability estimate of $f$, we need 
the dynamic data measured over the whole boundary $\partial B_R$. In contrast with the proof of Theorem \ref{Th2:Uniqueness} , we shall carry out the proof of Theorem \ref{t2} in the time domain without using the Fourier transform in the time variable.

\subsection{Proof of Theorem \ref{t2}}

As done in (\ref{decF}),
we can split $f$ via Hodge decomposition into the form
\be\label{dec}
f(\tilde{x},t)=
\begin{pmatrix}
\nabla_{\tilde{x}} f_p(\tilde{x},t) \\ 0
\end{pmatrix}
+
\begin{pmatrix}
\nabla_{\tilde{x}}^\perp f_s(\tilde{x},t) \\ 0
\end{pmatrix},
\en
where $f_p( \cdot,t)$ and $f_s( \cdot,t)$ are scalar functions compactly supported on $\tilde{B}_R$.
Fixing $\omega>0$ and $\xi\in\mathbb R^2$ such that
\begin{equation}\label{t2c} |\xi|^2>k_p^2:=\frac{\omega^2\rho}{\lambda+2\mu},\end{equation}
we introduce the time-dependent test function
\ben
V_p(x,t; \xi, \omega)&=&\begin{pmatrix}
-i\xi_1 \\ -i \xi_2 \\ \sqrt{|\xi|^2-k_p^2}
\end{pmatrix} e^{-i\xi\cdot \tilde{x}+ \sqrt{|\xi|^2-k_p^2}\; x_3}\,e^{-i\omega t}.\enn
In the same way, for
\begin{equation}\label{tt2c}|\xi|^2>k_s^2:=\frac{\omega^2\rho}{\mu},\end{equation}
we introduce the function
\ben V_s(x,t; \xi, \omega)&=&\begin{pmatrix}
i\xi_2 \\ -i \xi_1 \\ 0
\end{pmatrix} e^{-i\xi\cdot \tilde{x}+ \sqrt{|\xi|^2-k_s^2}\; x_3}\,e^{-i\omega t}.
\enn
Then, in a similar way to the proof of the uniqueness result, one can check that $V_\alpha$ ($\alpha=p,s$) are solutions to the homogeneous elastodynamic equation
\be\label{eq:V}
\rho\frac{\partial^2}{\partial t^2} V_\alpha(x,t; \xi, \omega)   -\mathcal{L}_{\lambda,\mu} V_\alpha(x,t; \xi, \omega)=0\qquad\mbox{in}\quad \R^3\times \R^+
\en for any fixed $\xi\in \R^2$ and $\omega\in \R$ satisfying \eqref{t2c} or \eqref{tt2c}.  Moreover, one can easily check that
\be\label{eq:VP}
\nabla_{\tilde{x}}^\perp\cdot V_p(x,t)=\nabla_{\tilde{x}}\cdot V_s(x,t)= 0,\quad (x,t)\in\R^3\times\R.
\en
Therefore, multiplying $V_p(x,t;\xi, \omega)$ to the right hand side of the equation (\ref{Lame2}), using (\ref{eq:V}) and applying integration by parts yield
 \ben
 &&\int_0^{T_1}\int_{B_{R_1}}\left(\rho\frac{\partial ^2}{\partial t^2} U(x,t)- \mathcal{L}_{\lambda,\mu} U(t, x)\right) \cdot V_p(x,t;\xi, \omega)\,dxdt \\
 &=&-\int_0^{T_1}\int_{\partial B_{R_1}} \mathcal T U(x,t)\cdot V_p(x,t;\xi, \omega)-  \mathcal T  V_p(t, x; \xi, \omega)\cdot U(x,t; \xi, \omega)\,ds(x)\,dt\\
 &&+ \int_0^{T_1}\int_{B_{R_1}}\left(\rho\frac{\partial ^2}{\partial t^2} U(x,t)\cdot  V_p(x,t;\xi, \omega)- \rho\frac{\partial ^2}{\partial t^2} V_p(x,t)\cdot  U(x,t;\xi, \omega)\right) \,dxdt \\
 &=&-\int_0^{T_1}\int_{\partial B_{R_1}} \mathcal T U(x,t)\cdot V_p(x,t;\xi, \omega)-  \mathcal T  V_p(t, x; \xi, \omega)\cdot U(x,t; \xi, \omega)\,ds(x)\,dt\\
 && +\;\rho\int_{B_{R_1}} \frac{\partial U(x,T_1)}{\partial t}\cdot V_p(x,T_1;\xi, \omega)-  \frac{\partial V_p(x,T_1;\xi, \omega)}{\partial t}\cdot U(x, T_1)\,dx.
 \enn
Again recalling Huygens principle, we know $U(x,T_1)=\partial_tU(x,T_1) =0$  for all $x\in B_R$ and $T_1>T+\frac{2R_1\sqrt{\rho}}{\sqrt{\mu}}$. Hence, the integral over $B_{R_1}$ on the right hand side of the previous identity vanishes.
 Following estimate \eqref{l2a} of Proposition \ref{p3} in the appendix, the traction of $U$ on the boundary $\partial B_{R_1}$ can be bounded by the trace of $U$ itself. Hence,  the left hand side can be bounded by
   \be\nonumber
&& \left|\int_0^{T_1}\int_{B_{R_1}}\left(\rho\frac{\partial ^2}{\partial t^2} U(x,t)- \mathcal{L}_{\lambda,\mu} U(t, x)\right) \cdot V_p(x,t;\xi, \omega)\,dxdt\right| \\ \nonumber
&=&\left| \int_0^{T_1}\int_{\partial B_{R_1}} \mathcal T U(x,t)\cdot V_p(x,t;\xi, \omega)-  \mathcal T  V_p(t, x; \xi, \omega)\cdot U(x,t; \xi, \omega)\,ds(x)\,dt \right| \\ \nonumber
 &\leq& || \mathcal T U||_{L^2((0,T_1)\times \partial B_{R_1})^3} ||V_p||_{L^2((0,T_1)\times \partial B_{R_1})^3}+
  ||U||_{L^2((0,T_1)\times \partial B_{R_1})^3} || \mathcal T V_p||_{L^2((0,T_1)\times \partial B_{R_1})^3}    \\ \nonumber
 &\leq& C\;\left(\norm{U}_{H^3(0,T_1;H^{3/2}(\partial B_{R_1})^3}||V_p||_{L^2((0,T_1)\times \partial B_{R_1})^3}+  ||U||_{L^2((0,T_1)\times \partial B_{R_1})^3}\; ||V_p||_{L^2(0,T_1;H^2(  B_{R_1}))} \right)\\ \nonumber
&\leq& C\;\norm{U}_{H^3(0,T_1;H^{3/2}(\partial B_{R_1})^3}||V_p||_{L^2(0,T_1;H^2(  B_{R_1}))}\\ \label{UP}
&\leq& C\;\norm{U}_{H^3(0,T_1;H^{3/2}(\partial B_{R_1})^3}(1+|(\xi,\omega)|^3)e^{R\sqrt{|\xi|^2-k_p^2} }
  \en for all $|\xi|>k_p$,
	where $C>0$ depends on $M$, $R_1$, $T_1$, $\rho$, $\lambda$ and $\mu$.
	On the other hand, using the governing equation (\ref{Lame2}) together with the relations (\ref{dec}), (\ref{eq:VP}) and using the fact that the sign of $g$ is constant,
we  obtain a lower bound of the left hand side of (\ref{UP}):
\ben
&&\left|\int_0^{T_1}\int_{B_{R_1}}\left(\rho\frac{\partial ^2}{\partial t^2} U(x,t)- \mathcal{L}_{\lambda,\mu} U(x,t)\right) \cdot V_p(x,t;\xi, \omega)\,dxdt\right| \\
&=& \left|\int_0^{T_1}\int_{B_{R_1}} f(\tilde{x},t)g(x_3) \cdot V_p(x,t;\xi, \omega)\,dx\,dt \right|\\
&=& \left|\int_0^{T_1}\int_{B_{R_1}} \begin{pmatrix}
\nabla_{\tilde{x}}f_p(\tilde{x},t)\\
0\end{pmatrix}
g(x_3)\;\cdot V_p(x,t;\xi, \omega)\,dx\,dt\right|\\
&=& |\xi|^2\,\left|\int_0^{T_1}\int_{B_{R_1}} f_p(\tilde{x},t) g(x_3)\,e^{-i\xi\cdot \tilde{x}+ \sqrt{|\xi|^2-k_p^2}\; x_3}\,e^{-i\omega t}\,  dx\,dt \right|\\
&=& |\xi|^2\,\left|\left(\int_0^{T_1}\int_{\tilde{B}_R} f_p(\tilde{x},t)e^{-i\xi\cdot \tilde{x}-i\omega t} d\tilde{x}dt\right)
\left(\int_{-R}^{R}
 g(x_3)\,e^{\sqrt{|\xi|^2-k_p^2}\; x_3}\,d x_3\right)\right|\\
 &\geq& |\xi|^2\,\left|\left(\int_0^{T_1}\int_{\tilde{B}_R} f_p(\tilde{x},t)e^{-i\xi\cdot \tilde{x}-i\omega t} d\tilde{x}dt\right)\right|
\norm{g}_{L^1(\R)} \; e^{-(\sqrt{|\xi|^2-k_p^2}) R},
 \enn for all $|\xi|>k_p$.
 Since $f_p$ is supported on $\tilde{B}_R\times(0,T_1)$,
 the first integral on the right hand of the last identity is the Fourier transform
 of $f_p$ with respect to $(\tilde{x},t)$ at the value $(\xi,\omega)$, which we denote by $\hat{f}_p(\xi,\omega)$.
Combining the previous two relations we obtain
\begin{equation}\label{es}
|\hat{f}_p(\xi,\omega)|\leq C\,\frac{(1+|(\xi,\omega)|^3)\norm{U}_{H^3(0,T;H^{3/2}(\partial B_{R_1}))^3}\,e^{2R\sqrt{|\xi|^2-k_p^2}}  }{|\xi|^2\; \norm{g}_{L^1(\R)} \;   }
\end{equation}
for all $|\xi|>k_p(\omega)$. We note that (\ref{es}) gives the estimate of $\hat{f}_p$ over the cone
$\{(\xi,\omega)\in\R^3: |\xi|^2>\omega^2 \rho/(\lambda+2\mu)\}$.
In order to derive from (\ref{es}) a stability estimate of $\hat{f}_p$ on $B_r$ for a large $r>0$,
 we will use a result of stability in the analytic continuation, following the arguments presented
 in \cite{CK,Ki1}. Below we state a stability estimate for analytic continuation problems; see \cite[Theorem 4]{AEWZ} (see also  \cite{Ma,Na}, where similar results  were established).
\begin{prop}\label{p2}  Let $s>0$ and assume that $g:\ B_{2s}\subset \R^{3}\to \mathbb C$ is a real analytic function satisfying
\[\norm{\nabla^{\beta }g}_{L^\infty(B_{2s})}\leq \frac{N\,\beta\,!}{(s \tau)^{\abs{\beta}}},\qquad \beta=(\beta_1,\beta_2,\beta_2)\in\mathbb N^{3},\]
for some $N>0$ and $0<\tau\leq1$. Further let $E\subset B_{s/2}$ be a  measurable set with strictly positive Lebesgue measure. Then,
\[\norm{ g}_{L^\infty(B_s)}\leq CN^{(1-b)}\norm{ g}_{L^\infty(E)}^{b},\]
where $b\in(0,1)$, $C>0$ depend  on   $\tau$, $\abs{E}$ and $s$.\end{prop}

Following \cite{Ki1}, we introduce the function
\[H_r(\xi,\omega):=\hat{f}_p(r(\xi,\omega))=(2\pi)^{-3/2}\int_{\R^{3}}f_p(\tilde{x},t)e^{-ir(\omega t+\xi\cdot \tilde{x})}d\tilde{x}dt\]
for some $r>1$ and $|(\xi,\omega)|\leq 2s$.
In a similar way to \cite{Ki1}, we fix $s=[\max(T_1,2R)]^{-1}+1$,  choose   $N=Ce^{3r}$, with $C$ some constant independent of $r$, and take $\tau=\frac{[\max(T_1,2R)]^{-1}}{s}=(s-1)/s$. Then we obtain
\begin{equation}\label{t2d} \norm{\partial_\omega^n\partial^\beta_\xi H_r}_{L^\infty(B_{2s})}\leq C\frac{e^{3r}\beta!n!}{([\max(T,2R)]^{-1})^{\abs{\beta}+n}}=\frac{N\beta!n!}{(s \tau)^{\abs{\beta}+n}},\quad n\in\mathbb N_+,\ \beta\in\mathbb N_+^{2}.\end{equation}
Moreover, fixing $c:=\frac{\rho}{\lambda+2\mu}$, $d:= \frac{s}{2\sqrt{1+c^{-1}}}$ and  $a_r\in\left(0,\frac{d}{\sqrt{c}}\right)$,  we define
$$E_r:=\left\{(\xi,\omega)\in\tilde{B}_d\in\times\left[-a_r,\frac{d}{\sqrt{c}}\right]:\ \max(r^{-2},\sqrt{c}|\omega|)<|\xi| \right\}.$$
It is easy to check that $E_r$ is a subset of $B_{s/2}$ in $\R^3$, and it is also a subset of the
cone
$\{(\xi,\omega)\in\R^3: |\xi|^2>\omega^2 \rho/(\lambda+2\mu)\}$.
We remark that $|E_r|=\kappa_r(-a_r)$, where
$$\kappa_r:y\mapsto \int_y^{\frac{d}{\sqrt{c}}}\int_{\max(r^{-2},\sqrt{c}|\omega|)<|\xi|<d}d\xi d\omega.$$
Note that $\kappa_r\left(-\frac{d}{\sqrt{c}}\right)=2\kappa_r(0)$ and one can check that
$$\kappa_r(0)=\frac{2\pi d^3}{3\sqrt{c}}+\pi r^{-2}\left(\frac{d^2}{3\sqrt{c}}-\frac{2r^{-4}}{3\sqrt{c}}\right).$$
Thus, there exists $r_0>1$ depending only on $R$, $\rho$, $\lambda$, $\mu$, $T$, such that
$$ \frac{\pi d^3}{2\sqrt{c}}<\kappa_r(0)< \frac{5\pi d^3}{6\sqrt{c}},\quad r>r_0.$$
Therefore, we have $$\kappa_r\left(-\frac{d}{\sqrt{c}}\right)=2\kappa_r(0)>\frac{\pi d^3}{\sqrt{c}}>\frac{5\pi d^3}{6\sqrt{c}}>\kappa_r(0)$$
and, from the continuity of the map $\kappa_r$, we deduce that we can choose $a_r$ in such way that
$$|E_r|=\kappa_r(a_r)=\frac{\pi d^3}{\sqrt{c}},\quad r>r_0.$$
This implies that, with such choice of $a_r$, the volume $|E_r|$ depends only on $R$, $\rho$, $\lambda$, $\mu$ and $T_1$. Consequently,  combining \eqref{t2d} with Proposition \ref{p2}, we deduce that
$$|\hat{f}_p(r(\xi,\omega))|=|H_r(\xi,\omega)|\leq Ce^{3(1-b)r}\left(\norm{ H_r}_{L^\infty(E_r)}\right)^{b},\quad |(\xi,\omega)|<s,\ r>r_0,$$
where $C>0$, $b\in(0,1)$ depend only on $R$, $\rho$, $\lambda$, $\mu$ and $T_1$.
In addition, applying \eqref{es}, we get
$$\norm{ H_r}_{L^\infty(E_r)}\leq Cr^4 e^{c_1 r}\norm{U}_{H^3(0,T;H^{3/2}(\partial B_R))^3},$$
where $C$ and $c_1$  depend only on $R$, $\rho$, $\lambda$, $\mu$ and $T_1$. Therefore, we can find $C$, $c$ depending only on $R$, $\rho$, $\lambda$, $\mu$ and $T_1$ such that
$$|\hat{f}_p(\xi,\omega)|\leq Ce^{cr}\norm{U}_{H^3(0,T;H^{3/2}(\partial B_R))^3},\quad |(\xi,\omega)|<r,\ r>sr_0.$$
 It follows that
\begin{equation}\label{t2e}\int_{B_r}|\hat{f}_p(\xi,\omega)|^2d\xi d\omega\leq Ce^{cr}\norm{U}^2_{H^3(0,T;H^{3/2}(\partial B_R))^3},\quad |(\xi,\omega)|<r,\ r>sr_0,\end{equation}
by eventually replacing the constants $C$ and $c$.  On the other hand, using \eqref{t2a} and the fact that $\Delta_{\tilde{x}}f_p=\nabla_{\tilde{x}}\cdot f$, we deduce that $f_p\in H^2(\R^3)$ and $\norm{f_p}_{H^2(\R^3)}\leq CM$, with $C$ depending only on $T$ and $R$. Thus, we find
$$\begin{aligned}\int_{|(\xi,\omega)|>r}|\hat{f}_p(\xi,\omega)|^2d\xi d\omega &\leq r^{-4}\int_{|(\xi,\omega)|>r}(1+|(\xi,\omega)|^4)|\hat{f}_p(\xi,\omega)|^2d\xi d\omega\\
\ &\leq r^{-4}\norm{f_p}^2_{H^2(\R^3)}\leq C^2r^{-4}M^2.\end{aligned}$$
Combining this with \eqref{t2e}, we find
$$\begin{aligned}\int_{\R^3}|\hat{f}_p(\xi,\omega)|^2d\xi d\omega&=\int_{B_r}|\hat{f}_p(\xi,\omega)|^2d\xi d\omega+\int_{|(\xi,\omega)|>r}|\hat{f}_p(\xi,\omega)|^2d\xi d\omega \\
\ &\leq C\left(e^{cr}\norm{U}^2_{H^3(0,T_1;H^{3/2}(\partial B_{R_1}))^3}+r^{-4}\right).\end{aligned}$$
Recalling the Plancherel formula, it holds that
\[
\norm{f_p}_{L^2((0,T_1)\times \tilde{B}_{R_1})}\leq C\,\left(e^{cr}\norm{U}_{H^3(0,T_1;H^{3/2}(\partial B_{R_1}))^3}+r^{-2}\right),\quad r>sr_0.
\]
Now, choosing $r=c^{-1}\ln(\norm{U}_{H^3(0,T_1;H^{3/2}(\partial B_{R_1}))^3})$, we get for $\norm{U}_{H^3(0,T_1;H^{3/2}(\partial B_{R_1}))^3}$ sufficiently small that
\begin{equation}\label{t2f}
\norm{f_p}_{L^2((0,T_1)\times \tilde{B}_{R_1})}\leq  C\,\left(\norm{U}^2_{H^3(0,T_1;H^{3/2}(\partial B_{R_1})^3}+\left|\ln\left(\norm{U}_{H^3(0,T_1;H^{3/2}(\partial B_{R_1})^3}\right)\right|^{-2}\right),\end{equation}
which can be obtained by
applying the classical arguments of optimization (see for instance the end of the proof of \cite[Theorem 1]{Ki1}). This gives the estimate of $f_p$ by our measurement data taken on $\partial B_{R_1}$.

Using similar arguments, we can prove
\begin{equation}\label{t2g}
\norm{f_s}_{L^2((0,T)\times \tilde{B}_R)}\leq  C\left(\norm{U}^2_{H^3(0,T_1;H^{3/2}(\partial B_R)^3}+\left|\ln\left(\norm{U}_{H^3(0,T_1;H^{3/2}(\partial B_R)^3}\right)\right|^{-2}\right).\end{equation}
On the other hand, by interpolation and the upper bound (\ref{t2a})
we have
$$\begin{aligned}\norm{f}_{L^2((0,T_1)\times \tilde{B}_R)}&\leq \norm{\nabla_{\tilde{x}}f_p}_{L^2((0,T_1)\times \tilde{B}_R)}+\norm{\nabla_{\tilde{x}}^\perp f_s}_{L^2((0,T_1)\times \tilde{B}_R)}\\
\ &\leq C(\norm{f_p}_{H^1((0,T_1)\times \tilde{B}_R)}+\norm{ f_s}_{H^1((0,T_1)\times \tilde{B}_R)})\\
\ &\leq C(\norm{f_p}_{H^2((0,T_1)\times \tilde{B}_R)}^{\frac{1}{2}}\norm{f_p}_{L^2((0,T_1)\times \tilde{B}_R)}^{\frac{1}{2}}+\norm{f_s}_{H^2((0,T_1)\times \tilde{B}_R)}^{\frac{1}{2}}\norm{f_s}_{L^2((0,T_1)\times \tilde{B}_R)}^{\frac{1}{2}})\\
\ &\leq C(\norm{f_p}_{L^2((0,T_1)\times \tilde{B}_R)}^{\frac{1}{2}}+\norm{f_s}_{L^2((0,T_1)\times \tilde{B}_R)}^{\frac{1}{2}}),\end{aligned}$$
with $C$ depending on $M$, $T_1$ and $R$. Then, combining this with \eqref{t2f}-\eqref{t2g}, we obtain \eqref{t2b}.

\begin{rem}\label{remark1} The uniqueness and stability results presented in Theorems \ref{Th2:Uniqueness} and \ref{t2} carry over to the
scalar inhomogeneous wave equation of the form
\ben
\frac{1}{c^2}\partial_{tt}U(x,t)=\Delta U(x,t)+f(\tilde{x},t)g(x_3), && (x,t)\in \R^3\times (0,\infty),\\
U(x,0)=U_t(x,0)=0, && x\in \R^3,
\enn
where both $f$ and $g$ are compactly supported scalar functions and $c$ a constant.   If the wave speed $c$ and the function $g(x_3)$ are known, one can determine the source term $f(\tilde{x},t)$ from partial boundary data. In particular, $f$ is allowed to be a moving source with the orbit lying on the $ox_1x_2$-plane.
 In the frequency domain, the above wave equation gives rise to an inverse problem of recovering the wave-number-dependent source term $f(\tilde{x}, k)$ from the multi-frequency boundary observation data of the inhomogeneous Helmholtz equation
\ben
\Delta u(x, k)+k^2/c^2\; u(x,k)=\hat{f}(\tilde{x},k)g(x_3).
\enn
Progress along these directions will be reported in our forthcoming publications.
\end{rem}

\section{Appendix}

\subsection{Well-posedness result and estimation of surface traction}

In this subsection, we consider the inhomogeneous Lam\'e system
\be\label{eq1}
\left\{\begin{array}{ll}\rho\partial_t^2U- \mathcal{L}_{\lambda,\mu} U=F(x,t),\quad &(x,t)\in\R^{3}\times(0,+\infty),\\  U(\cdot,0)= \partial_tU(\cdot,0)=0,\quad & x\in \R^3,\end{array}\right.\en
where the operator $\mathcal{L}_{\lambda,\mu}$ is given by (\ref{lame}).
We assume that supp$(F)\subset [0,T)\times B_R$, with $B_R:=\{x\in\R^3:\ |x|<R\}$.
It is well-known that the operator $\mathcal{L}_{\lambda,\mu}$ is an elliptic operator and the standard elliptic regularity holds; see e.g., \cite[Chapters 4 and 10]{Mclean} and \cite[Chapter 5]{Hsiao}.
The quadratic form corresponding to $\mathcal{L}_{\lambda,\mu}$ is given by
\ben
\mathcal{E}(U, V):=\lambda\, (\Div U)(\Div V)+2\mu E(U):E(V)
\enn
where the stress tensor $E$ is defined via (\ref{E}), with the notation  $A:B:=\sum_{i,j=1}^3 a_{ij} b_{ij}$ for $A=(a_{ij})_{i,j=1}^3$, $B=(b_{ij})_{i,j=1}^3$.
Hence, for a bounded Lipschitz domain $D\subset \R^3$ there holds the relation (see e.g., \cite[Lemma 3]{AAIY})
\be\label{Betti}
-\int_{D}\mathcal{L}_{\lambda,\mu}U  \cdot \overline{V} \, dx=
\int_{D}\mathcal{E}(U,\overline{V}) \, dx-\int_{\partial D}\overline{V} \cdot \mathcal T U \, ds
\en
for all $U, V\in H^2(D)^3$.
By the well-known Korn's inequality
(see e.g. \cite[Theorem 10.2]{Mclean}, \cite[Chapter 3]
{DuvautLions}), it holds that
\begin{equation}\label{qad}
\int_{\R^3}\mathcal{E}(U,\overline{U})+c_1\, \norm{U}_{L^2(\R^3)^3}\geq c_2\norm{U}_{H^1(\R^3)^3},\quad U\in H^1(\R^3)^3
\end{equation}
for some constants $c_1,c_2>0$. In the particular case of constant Lam\'e coefficients, we have
\ben
\mathcal{L}_{\lambda,\mu} U=\mu\Delta U +(\lambda+2\mu) \nabla (\nabla\cdot U),
\enn
and
\ben
\mathcal{E}(U, V)
=2\mu\sum_{j,k=1}^3\partial_k U_j\,\, \partial_k V_j+\lambda\, (\Div U)(\Div V)-\mu\,\curl U\cdot\curl V.
\enn
In this case, the surface traction can be simplified to be
\ben
 \mathcal T U= 2 \mu \, \partial_{\nu\,} U + \lambda(\Div U)\,\nu\,
+\mu\, \nu\times \curl\,U\qquad\mbox{on}\quad\partial D.
\enn
We refer to the monograph \cite{Ku} for comprehensive studies on the Lam\'e system.
Below we state a well-posedness result to the elastodynamic system in unbounded domains by applying the standard arguments of \cite[Chapter 8]{LM1}.
\begin{lem}\label{l1} Let $F\in H^1(\R_+;L^2(\R^3))^3$. Then
 problem \eqref{eq1} admits a unique solution
 $$U\in \mathcal C^1([0,+\infty);L^2(\R^3))^3\cap \mathcal C([0,+\infty);H^1(\R^3))^3.$$ Under the additional
 condition $F(\cdot,0)=x\mapsto F(0,x)=0$, the unique solution lies in the space $$U\in \mathcal C^2([0,+\infty);L^2(\R^3))^3\cap \mathcal C([0,+\infty);H^2(\R^3))^3.$$
\end{lem}

\begin{proof} Without lost of generality, we assume that $\rho=1$. We define on $L^2(\R^3)^3$ the sesquilinear form $a$ with domain $D(a):=H^1(\R^3)^3$ given by
$$a(U_1,U_2):=\int_{\R^3}\mathcal{E}(U_1,U_2)dx.$$
In view of \eqref{Betti}, by density, we find
$$\left\langle -\mathcal{L}_{\lambda,\mu}U_1,U_2 \right\rangle_{H^{-1}(\R^3)^3,H^1(\R^3)^3}=a(U_1,U_2),\quad U_1,U_2\in H^1(\R^3)^3.$$
Therefore, in view of \eqref{qad}, fixing $H=L^2(\R^3)^3$, $V=H^1(\R^3)^3$ and applying \cite[Theorem 8.1, Chapter 3]{LM1} and \cite[Theorem 8.2, Chapter 3]{LM1}, we deduce that \eqref{eq1} admits a unique solution  $U\in \mathcal C^1([0,+\infty);L^2(\R^3))^3\cap \mathcal C([0,+\infty);H^1(\R^3))^3$. It remains to prove that $U\in \mathcal C^2([0,+\infty);L^2(\R^3))^3\cap \mathcal C([0,+\infty);H^2(\R^3))^3$.
For this purpose, we consider $V=\partial_tU$ and, using the fact that $F(\cdot,0)=0$,
we deduce that $V$ solves
\begin{equation}\label{eq2}\left\{\begin{array}{ll}\partial_t^2V-\mathcal L_{\lambda,\mu} V=\partial_t\tilde{F}(x,t),\quad &(x,t)\in\R^{3}\times(0,+\infty),\\  V(\cdot,0)= \partial_tV(\cdot,0)=0,\quad &x\in \R^3.\end{array}\right.\end{equation}
Using the fact that $\partial_tF\in L^2((0,+\infty)\times\R^3)$ and applying the above arguments we deduce that $V\in \mathcal C^1([0,+\infty);L^2(\R^3))^3$ and  that $U\in \mathcal C^2([0,+\infty);L^2(\R^3))^3$. Therefore, for any $t\in[0,+\infty)$, $U$ is a solution of the boundary value problem
\begin{equation}\label{eq3}-\mathcal L_{\lambda,\mu} U(x,t)=-\partial_t^2U(x,t)+F(x,t),\quad x\in\R^{3}.\end{equation}
Since $\partial_t^2U(t,\cdot)\in L^2(\R^3)^3$, from the elliptic regularity of the operator $-\mathcal L_{\lambda,\mu}$ (see e.g. \cite[Theorem 5.8.1]{Hsiao}), we deduce that $U(t,\cdot)\in H^2(\R^3)^3$. Moreover, for any $t_1,t_2\in [0,+\infty)$, we have
$$\begin{aligned}&\norm{U(\cdot,t_1)-U(\cdot,t_2)}_{H^2(\R^3)^3}\\
&\leq C(\norm{\mathcal L_{\lambda,\mu} (U(\cdot,t_1)-U(\cdot,t_2))}_{L^2(\R^3)^3}+\norm{U(\cdot,t_1)-U(\cdot,t_2)}_{L^2(\R^3)^3})\\ \ &\leq C\left(\norm{\partial_t^2U(\cdot,t_1)-\partial_t^2U(\cdot,t_2))}_{L^2(\R^3)^3}+\norm{U(\cdot,t_1)-U(\cdot,t_2)}_{L^2(\R^3)^3}+\norm{F(\cdot,t_1)-F(\cdot,t_2))}_{L^2(\R^3)^3}\right).\end{aligned}$$
Therefore, using the fact that $U\in \mathcal C^2([0,+\infty);L^2(\R^3))^3$ and the fact that $F$ extended by $0$ to $\R^3\times\R$ is lying in $ H^1(\R;L^2(\R^3))^3\subset \mathcal C(\R;L^2(\R^3))^3$, we deduce that $U\in \mathcal C([0,+\infty);H^2(\R^3))^3$. \end{proof}

Using this result for $F\in H^1(0,T;L^2(\R^3))^3$, we know
 $$U(x,t),\;
 \mathcal T  U(x,t)\in \mathcal C([0,+\infty);L^2(\partial B_R))^3, \quad (x,t)\in \partial B_R\times[0,\infty).$$
With additional smoothness assumptions we can also estimate $T U_{|\partial B_R\times[0,+\infty)}$ by $U_{|\partial B_R\times[0,+\infty)}$.
The main result of this subsection can be stated as follows.
\begin{prop}\label{p3} Let $T_2>0$ and  let $F\in H^4(\R_+;L^2(\R^3))^3$ be such that $F(\cdot,0)=\partial_t F(\cdot,0)=\partial_t^2 F(\cdot,0)=\partial_t^3 F(\cdot,0)=0$. Then  problem \eqref{eq1} admits a unique solution $U\in \mathcal C^4([0,+\infty);L^2(\R^3))^3\cap \mathcal C^3([0,+\infty);H^2(\R^3))^3$ satisfying the estimate
\begin{equation}\label{l2a}\norm{ \mathcal T  U}_{L^2(\partial B_R\times(0,T_2))^3}\leq C\norm{ U}_{H^3(0,T_2;H^{\frac{3}{2}}(\partial B_R))^3},\end{equation}
with $C$ depending on $\lambda$, $\rho$, $\mu$, $T_2$ and $R$.\end{prop}

\begin{proof} Note first that $W=\partial_t^3U$ solves
\begin{equation}\label{eq4}\left\{\begin{array}{ll}\rho\partial_t^2W-\mathcal L_{\lambda,\mu} W=\partial_t^3F(x,t), & (x,t)\in\R^3\times(0,+\infty),\\  W(\cdot,0)= \partial_tW(\cdot,0)=0,\quad & x\in \R^3.\end{array}\right.\end{equation}
Using the fact that $\partial_t^3 F\in H^1(\R_+;L^2(\R^3))^3$ with $\partial_t^3 F(\cdot,0)=0$, we can apply Lemma \ref{l1} to deduce that $W\in \mathcal C^2([0,+\infty);L^2(\R^3))^3\cap \mathcal C([0,+\infty);H^2(\R^3))^3$. Thus, we have   $U\in \mathcal C^4([0,+\infty);L^2(\R^3))^3\cap \mathcal C^3([0,+\infty);H^2(\R^3))^3$. This implies that $g:=U_{|\partial B_R\times[0,+\infty)}\in\mathcal C^3([0,+\infty);H^{\frac{3}{2}}(\partial B_R))^3$. Hence, the restriction of $U$ to $(\R^3\backslash B_R)\times (0,T_2)$ solves the initial boundary value problem
\begin{equation}\label{eq5}\left\{\begin{array}{ll}\rho\partial_t^2U-\mathcal L_{\lambda,\mu} U=0,\quad &(x,t)\in(\R^3\backslash B_R)\times(0,T_2),\\  U(\cdot,0)= \partial_tU(\cdot,0)=0,\quad &x\in \R^3\backslash B_R,\\
U=g,\quad&(x,t)\in \partial B_R\times(0,T_2).\end{array}\right.\end{equation}
Using the fact that $g(\cdot,0)=\partial_tg(\cdot,0)=\partial_t^2g(\cdot,0)=0$, from a classical lifting result, one can find $G\in \mathcal C^3([0,+\infty);H^2(\R^3\backslash B_R))^3$ such that $G_{|\partial B_R\times(0,T_2) }=g$, $G(\cdot,0)=\partial_tG(\cdot,0)=\partial_t^2G(\cdot,0)=0$ and
\begin{equation}\label{l2b}\norm{G}_{H^3(0,T_2;H^2(\R^3))^3}\leq C\norm{g}_{H^3(0,T_2;H^{\frac{3}{2}}(\partial B_R))^3},\end{equation}
with $C$ depending only on $T_2$ and $R$. Therefore, we can split $U$ to $U=V+G$ on $(0,T_2)\times (\R^3\backslash B_R)$, with $V$ the solution of
$$\left\{\begin{array}{ll}\rho\partial_t^2V-\mathcal L_{\lambda,\mu} V=-(\rho\partial_t^2G-\mathcal L_{\lambda,\mu}G):=H,\quad &(x,t)\in(\R^3\backslash B_R)\times(0,T_2),\\  V(\cdot,0)= \partial_tV(\cdot,0)=0,\quad & x\in \R^3\backslash B_R,\\
V=0, &(x,t)\in \partial B_R\times (0,T_2).\end{array}\right.$$

Using the fact that $H\in H^1(0,T_2;L^2( \R^3\backslash B_R))$ and $H(\cdot,0)=0$, in a similar way to Lemma \ref{l1} we can prove that $V\in \mathcal C^2([0,T_2];L^2(\R^3\backslash B_R))^3\cap \mathcal C([0,T_2];H^2(\R^3\backslash B_R))^3$, with
$$\norm{V}_{L^2(0,T_2;H^2(\R^3\backslash B_R))^3}\leq C\norm{H}_{H^1(0,T_2;L^2(\R^3\backslash B_R))^3}\leq C\norm{G}_{H^3(0,T_2;H^2(\R^3))^3},$$
with $C$ depending on $\lambda$, $\rho$, $\mu$, $T_2$ and $R$. Combining this with \eqref{l2b}, we deduce that
$$\norm{U}_{L^2(0,T_2;H^2(\R^3\backslash B_R))^3}\leq C\norm{g}_{H^3(0,T_2;H^{\frac{3}{2}}(\partial B_R))^3}$$
and using the continuity of the trace map, 
we obtain
$$\norm{\mathcal T U}_{L^2((0,T_2)\times\partial B_R)^3}\leq C\norm{U}_{L^2(0,T_2;H^2(\R^3\backslash B_R))^3}.$$
Combining the last two estimates we finally obtain \eqref{l2a}.\end{proof}

\subsection{Long time asymptotic behavior of the solution on a bounded domain}
In this subsection we fix $\Omega_1$  a bounded $\mathcal C^2$ domain of $\R^3$. We consider the bilinear form $a$ with domain $D(a)=H^1(\Omega_1)^3$ given by
\ben a(U_1,U_2)&=&\int_{\Omega_1}\mathcal{E}(U,\overline{V}) \, dx\\
&=&\int_{\Omega_1} \left[(\lambda(x)+2\mu(x))(\nabla\cdot U_1(x))\overline{(\nabla\cdot U_2(x))}-\mu(x)\nabla\times U_1(x)\cdot\overline{\nabla\times U_2(x)}\right]dx.\enn
Then,  for $U_1\in H^1(\Omega_1)^3$ such that $\mathcal{L}_{\lambda,\mu}U_1\in L^2(\Omega_1)^3$ and $ \mathcal T U_1=0$ on $\partial\Omega_1$, we have
\be\label{sesa}
a(U_1,U_2)=-\int_{\Omega_1} \mathcal{L}_{\lambda,\mu}U_1(x)\cdot \overline{U_2(x)}dx.
\en
Fixing $H=L^2(\Omega_1)^3$, $V=H^1(\Omega_1)^3$ and applying \cite[Chapter 5]{Hsiao}, \cite[Theorem 8.1, Chapter 3]{LM1} and \cite[Theorem 8.2, Chapter 3]{LM1}, we deduce that, for $F\in L^2((0,+\infty)\times \Omega_1)^3$, $V_0\in V$ and $V_1\in H$, the problem

\be\label{eaq1}\left\{\begin{array}{lll}\rho\partial_t^2U-\mathcal{L}_{\lambda,\mu} U=F(x,t),\quad &&(x,t)\in\Omega_1\times(0,+\infty),\\  U(\cdot,0)=V_0,\  \partial_tU(\cdot,0)=V_1,\quad && x\in \Omega_1\\
 \mathcal T U(x,t)=0,\quad && (x,t)\in \partial \Omega_1\times(0,+\infty),
\end{array}\right.\en
admits a unique solution in $\mathcal C^1([0,+\infty);L^2(\Omega_1))^3\cap \mathcal C([0,+\infty);H^1(\Omega_1))^3$. Below we show the long time behavior of the solution of \eqref{eaq1}.

\begin{prop}\label{p4}
Let $F\in L^2((0,+\infty)\times \Omega_1)^3$ be such that supp$(F)\subset \overline{\Omega_1}\times[0,T)$ and let $V_0\in H^1(\Omega_1)^3$, $V_1\in L^2(\Omega)^3$. Then  problem \eqref{eaq1} admits a unique solution $U\in \mathcal C^1([0,+\infty);L^2(\Omega_1))^3\cap \mathcal C([0,+\infty);H^1(\Omega_1))^3$ satisfying
\begin{equation}\label{p4a}\norm{U(t,\cdot)}_{H^1(\Omega_1)^3}\leq \tilde{C}(\norm{F}_{L^2((0,T)\times\Omega_1)} +\norm{V_0}_{H^1(\Omega_1)^3}+\norm{V_1}_{L^2(\Omega_1)^3})(t+1),\quad t>0,\end{equation}
with $\tilde{C}$ independent of $t$.
\end{prop}

\begin{proof} Without lost of generality we may assume that $V_0=V_1=0$. Indeed the result with non-vanishing initial conditions can be carry out in a similar way.
Let us first assume that $F\in H^1_0(0,T;L^2(\Omega_1))^3$.
Repeating the arguments in the proof of Lemma \ref{l1}, we can prove that the regularity of $U$ can be improved to be
$$U\in \mathcal C^2([0,+\infty);L^2(\Omega_1))^3\cap \mathcal C([0,+\infty);H^2(\Omega_1))^3.$$
Now let us consider the energy
$$J(t):=\int_{\Omega_1} \rho|\partial_t U(x,t)|^2+\mathcal E(U(x,t),U(x,t))\;dx.$$
 For simplicity, we  assume  that $F$ takes values in $\R^3$ such that $U$ takes also values in $\R^3$, otherwise our arguments may be extended without any difficulty to function $F$ taking values in $\mathbb C^3$. It is clear that $J\in \mathcal C^1([0,+\infty))$ and
$$J'(t)=2\int_{\Omega_1} \rho\partial_t^2 U(x,t)\cdot \partial_t U(x,t) +\mathcal E(U(x,t),\partial_tU(x,t))\;dx.$$
Using the fact that $ \mathcal T U=0$ on $[0,+\infty)\times\partial\Omega_1$, we can integrate by parts to obtain (see (\ref{sesa}))
$$J'(t)=2\int_{\Omega_1} [\rho\partial_t^2 U(x,t)-\mathcal{L}_{\lambda,\mu} U(x,t)]\cdot \partial_t U(x,t)dx=2\int_{\Omega_1} F(x,t)\cdot \partial_tU(x,t).$$
Thus, using the fact that $\rho>0$ we get
\be \nonumber\int_{\Omega_1} |\partial_t U(x,t)|^2dx\leq \rho^{-1}J(t)&\leq& 2\rho^{-1}\int_0^t\abs{\int_{\Omega_1} F(x,s)\cdot \partial_tU(x,s)dx}ds\\ \nonumber
 &\leq& 2\rho^{-1}\int_0^T\int_{\Omega_1}|F(x,s)|\;|\partial_tU(x,s)|dxds\\ \label{eq:U}
&\leq& 2\rho^{-1} T^{\frac{1}{2}}\norm{F}_{L^2((0,T)\times\Omega_1)^3}\norm{\partial_tU}_{L^\infty(0,T;L^2(\Omega_1))^3}.
\en
Combining this with a classical estimate of $\norm{\partial_tU}_{L^\infty(0,T;L^2(\Omega_1))^3}$ (e.g. \cite[Formula (8.15), Chapter 3]{LM1}), we deduce that
$$\norm{\partial_tU}_{L^\infty(0,T;L^2(\Omega_1))^3}\leq \tilde{C}\norm{F}_{L^2((0,T)\times\Omega_1)^3}$$
with $\tilde{C}$ depending only on $T$, $\Omega_1$, $\rho$, $\mathcal{L}_{\lambda,\mu}$. It then follows from (\ref{eq:U})  that
$$\int_{\Omega_1} |\partial_t U(x,t)|^2dx\leq \tilde{C}\norm{F}_{L^2((0,T)\times\Omega_1)^3}^2.$$
Combining this with the fact that
$$U(t,\cdot)=\int_0^t\partial_t U(\cdot,s)ds,$$
we obtain that
\be\nonumber
\norm{U(\cdot,t)}_{L^2(\Omega_1)^3}&=&\norm{\int_0^t\partial_t U(\cdot,s)ds}_{L^2(\Omega_1)^3}\leq \int_0^t\norm{\partial_t U(\cdot,s)}_{L^2(\Omega_1)^3}ds\\ \label{eq:U1}
&\leq& \tilde{C}\,\norm{F}_{L^2((0,T)\times\Omega_1)^3}t.
\en
 By density, we can extend this estimate  to $F\in L^2((0,T_1)\times\Omega_1)^3$.

 Applying Korn's inequality gives the estimate
\be\nonumber
\norm{U(\cdot,t)}_{H^1(\Omega_1)^3}^2&\leq& \tilde{C}\left(\norm{U(\cdot,t)}_{L^2(\Omega_1)^3}^2+\int_{\Omega_1} \mathcal E(U(x,t),U(x,t))dx\right)\\ \label{EU}
&\leq& \tilde{C}(\norm{U(\cdot,t)}_{L^2(\Omega_1)^3}^2+J(t)).
\en
Multiplying $U$ to both sides of (\ref{eaq1}) and  integrating by part with respect to $x$ over $\Omega_1$, we can estimate $J(t)$ by
\be\label{EJ}
J(t)=\int_{\Omega_1} F(x,t)\,\cdot U(x,t)\,dx\leq ||F(\cdot,t)||_{L^2(\Omega_1)^3}||U(\cdot,t)||_{L^2(\Omega_1)^3}.
\en
Now, inserting (\ref{EJ}) into (\ref{EU}) and making use of (\ref{eq:U1}), we finally obtain
\ben
\norm{U(\cdot,t)}_{H^1(\Omega_1)^3}^2\leq \tilde{C}\,\norm{F}_{L^2((0,T)\times\Omega_1)^3}^2(1+t^2),
\enn
which proves (\ref{p4a}).
\end{proof}

\section*{Acknowledgement}
 The work of G. Hu is supported by the NSFC grant (No. 11671028) and NSAF grant (No. U1530401).  The work of Y. Kian is supported by  the French National
Research Agency ANR (project MultiOnde) grant ANR-17-CE40-0029. The  authors would like to
thank Yikan Liu for his remarks that allow to improve Theorem 3. The second author is grateful to Lauri Oksanen for his remarks about the global Holmgren theorem for Lam\'e system that allows to improve our results related to (IP1). The second author would like
to thank the Beijing Computational Science Research Center, where part of this article was written, for its
kind hospitality

\end{document}